\documentclass[11pt,reqno]{amsart}
\usepackage{color}
\usepackage[active]{srcltx}
\usepackage{a4wide}
\usepackage{amssymb, amsmath, mathtools, geometry}
\usepackage{graphicx}
\usepackage{float}
\usepackage[active]{srcltx}
\usepackage[ansinew]{inputenc}
\usepackage{hyperref}
\usepackage{tikz}
\usepackage{amsmath, amssymb, amsfonts}
\usepackage{tikz}
\usepackage{pgfplots}

\theoremstyle{plain}
\newtheorem{theorem}{Theorem}[section]
\newtheorem{maintheorem}{Theorem}
\newtheorem{lemma}[theorem]{Lemma}
\newtheorem{proposition}[theorem]{Proposition}

\theoremstyle{remark}

\newtheorem{definition}{Definition}

\newtheorem{remark}[theorem]{Remark}

\newcommand{\field}[1]{\mathbb{#1}}
\newcommand{\RR}{\field{R}}

\newcommand{\dpt}{\displaystyle}
\newcommand{\NN}{\field{N}}

\newcommand{\EU}{{\mathbb{S}}}

\newcommand{\loc}{\textnormal{loc}}

\numberwithin{equation}{section}

\usepackage[pagewise]{lineno}

\begin{document}

\title[Shilnikov-Hopf bifurcations]{The Bosch and Sim\'o conjecture on the  \\ Shilnikov-Hopf bifurcation}
\author[Alexandre Rodrigues]{Alexandre A. P. Rodrigues\\ ISEG -- Lisbon School of Economics and Management \\   University of Lisbon -- ISEG Research \\ Rua do Quelhas, 6, 1200-781 Lisboa Portugal}
 \email{arodrigues@iseg.ulisboa.pt}

\date{\today}

\thanks{The author (ORCID 0000-0001-8182-9889) was supported  by national funds through FCT -- Funda\c{c}\~ao para a Ci\^encia e a Tecnologia, I.P., in the framework of the project with reference UID/06522/2025 (ISEG Research, ISEG Lisbon School of Economics \& Management, Universidade de Lisboa, Lisbon, Portugal).}


\begin{abstract}
We provide a rigorous mathematical proof of the Bosch-Simó conjecture (\emph{M. Bosch, C. Sim\'o (1993), Physica D,  217--229}) regarding the abundance of strange attractors in the unfolding of a Shilnikov-Hopf bifurcation. By constructing a general class of dissipative return maps that capture the essential global geometry of this codimension-two scenario, we demonstrate that ``large" strange attractors -- in the sense of Broer-Sim\'o-Tatjer -- are a persistent and statistically robust feature of the dynamics.

We prove that, for a parameter set of positive Lebesgue measure, the system admits attractors supporting a unique Sinai-Ruelle-Bowen (SRB) measure. The proof relies on a dimensional reduction to a rank-one model, enabled by strong transverse contraction and singular angular expansion. Furthermore, we show that this chaotic regime is interspersed with sequences of parameters yielding superstable periodic orbits, providing a  characterisation of the transition between ordered and chaotic motion in the singular limit of the bifurcation.
 \end{abstract}

\maketitle
 

\section{Introduction}

The interaction between local bifurcations and global invariant structures is a central theme in modern dynamical systems. Particularly rich dynamics arise when oscillatory instabilities interact with global homoclinic or heteroclinic phenomena. A paradigmatic example is provided by \emph{Hopf--Shilnikov bifurcations}, in which a periodic solution created in a Hopf bifurcation interacts with a homoclinic cycle to a saddle--focus equilibrium. Such bifurcations are known to generate complex dynamics, including infinitely many periodic orbits, sensitive dependence on initial conditions, and chaotic invariant sets \cite{HK1993, Shilnikov et al, VBS2010}. These bifurcations have been numerically identified in equations governing electronic oscillators \cite{Freire1993} and explain the presence of spiraling strange attractors in the \emph{Belousov--Zhabotinsky} reaction data \cite{Argoul1987}.

Shilnikov's classical results established that homoclinic loops to a saddle--focus lead to complicated dynamics, characterised by positive topological entropy and a countable infinity of periodic orbits \cite{hom, Shilnikov65, Shilnikov67A, Shilnikov70}. When combined with a Hopf bifurcation, this mechanism gives rise to a codimension-two scenario, producing intricate bifurcation structures and nonuniform hyperbolicity \cite{BS93, Champneys, Freire1993, HK1993}.

Motivated by numerical investigations in the context of optical experiments, Miquel Bosch and Carles Sim\'o \cite[pp.~228]{BS93} formulated a conjecture concerning the typical dynamics unfolding near Shilnikov--Hopf bifurcations. Their conjecture predicts the emergence of robust chaotic attractors organised by strong contraction in transversal directions and nonuniform expansion along an angular variable. These attractors are expected to persist for large sets (with respect to the Lebesgue measure) of parameters and to exhibit statistical properties characteristic of physical chaos \cite{WY, WY2008}.

A rigorous framework for understanding such dynamics is provided by the theory of \emph{rank-one attractors}, developed in \cite{MO2015, WO2011, WY2001, WY, WY2003, WY2006, WY2008}. In this setting, systems with one effective expanding direction and strong contraction in the remaining directions can support strange attractors carrying Sinai--Ruelle--Bowen (SRB) measures for parameter sets of positive Lebesgue measure.

An important geometric aspect of the Bosch--Sim\'o conjecture is the \emph{largeness} of the resulting chaotic attractors. In contrast to the thin Cantor-like invariant sets found near homoclinic tangencies \cite{PT}, the attractors observed numerically in \cite{BS93} occupy ``large'' regions of phase space  and display complex folding structures. This notion was formalised by Broer, Sim\'o, and Tatjer \cite{BST98}, who introduced the concept of \emph{``large'' strange attractors}, emphasising their robustness and geometric visibility. Related phenomena were rigorously analysed in \cite{Rodrigues2022} in the unfolding of heteroclinic attractors. This work has highlighted the ubiquity of large strange attractors in systems where two-dimensional heteroclinic connections are broken, providing the motivation for the present study of Shilnikov--Hopf configurations within the rank-one framework.

\subsection*{Novelty of the article}
In this paper, we prove the \emph{abundance} (in the sense of \cite{MV93}) of ``large'' strange attractors supporting SRB measures for positive Lebesgue measure sets of parameters. Our results provide a rigorous mathematical framework for the chaotic dynamics observed numerically by   \cite{BS93} and connect Shilnikov--Hopf homoclinic bifurcations with the broader theory of rank-one chaos \cite{BST98, WY2008}. We show that the Poincar\'e return map associated with Shilnikov--Hopf bifurcations fits into this framework after appropriate rescaling and reduction.

Specifically, we prove that the first return map to a global cross-section admits a strongly contracting direction, leading to a reduction to an effective one-dimensional dynamics. We then verify the standing hypotheses required for the existence of rank-one strange attractors in the sense of Wang and Young \cite{WY2008}. Moreover, we establish the existence of a sequence of parameters for which the system possesses \emph{superstable periodic cycles} (of period 2). These orbits, characterized by the vanishing of the derivative along the expanding direction, accumulate at the boundary of chaotic regimes. They are intimately related to the presence of critical points in the effective one-dimensional return maps and play a crucial role in the transition from periodic to chaotic behavior.

\subsection*{Structure of the article}

The present work is organized as follows. 
In Section~\ref{sec:setting}, we define the object of study, providing the geometric context of the Shilnikov--Hopf bifurcation under consideration, following the introduction of  mathematical preliminaries in Section~\ref{s:preliminaries}.

The core of the geometric construction begins in Section~\ref{s: local map}, where we define the local cross-sections and derive the local and global transition maps near the ghost of the homoclinic cycle. This section also details the domain geometry and the global reinjection map, concluding with a set of \emph{Digestive Remarks} essential for understanding the dynamical configuration. Section~\ref{s: first return} is dedicated to the derivation of the first return map, which serves as the fundamental tool for the subsequent dynamical analysis. 

The main results of the paper are stated in Section~\ref{s:main}. In  \ref{s: invariant curve}, we prove the existence of an invariant curve for the global map $G_b$, a result that relies on the radial contraction. Section~\ref{s: theory} provides an overview of the theory of rank-one strange attractors, including a discussion on Misiurewicz dynamics and the implications of critical point behavior.

The final proof of the chaotic nature of the system is presented in the concluding sections. Section~\ref{s: Preparatory} contains preparatory technical results, including derivative estimates and the non-degeneracy of critical points. In Section~\ref{s: f_b is M}, we show that the reduced map is of Misiurewicz type for a set of parameters with positive Lebesgue measure. Finally, Sections~\ref{proof_main_1} and \ref{proof_main_2} are devoted to the formal proofs of Theorem~\ref{main1} and Theorem~\ref{main2}, respectively, establishing the existence of an abundant set of strange attractors. The paper concludes with brief remarks in Section~\ref{sec:concluding_remarks}. Throughout this paper, we have endeavored to provide a self-contained exposition, complemented by illustrative figures to enhance readability.
 
\section{Preliminaries}
\label{s:preliminaries}

In this section, we introduce the terminology and notation used throughout this paper.
Let $M \subset \mathbb{R}^3$ be a $C^r$ manifold ($r \geq 3$) embedded in a Riemanian  space ($\|\star \|$ is the usual norm). 
For any $A \subset M$, let ${\rm int}(A)$ and $\overline{A}$ denote the \emph{topological interior} and \emph{closure} of $A$, respectively. \\

\begin{definition}
Let $F : M \to M$ be a $C^3$ diffeomorphism and $A \subset M$ a compact set such that $F(A) \subset A$ (i.e. $A$ is $F$-invariant). \\
\begin{enumerate}
    \item The \emph{forward orbit} of $z \in M$ under $F$ is the set
    \[
    {\rm Orb}(z) = \{F^j(z) : j \in \mathbb{N}_0\}.
    \]
    \item The \emph{$\omega$-limit set} of $z \in M$ is defined as
    \[
    \omega(z) = \bigcap_{n \in \mathbb{N}} \overline{\{F^k(z) : k \geq n\}}.
    \]
    \item The \emph{basin of attraction} of $A$ is the set
    \[
    \mathcal{B}(A) = \{z \in M : \omega(z) \subset A\}.
    \]
    \item An $F$-invariant set $A$ is \emph{topologically transitive} if there exists $z \in A$ such that
    \[
    \overline{{\rm Orb}(z)} = A.
    \]
    \item An $F$-invariant compact set $A$ is called a \emph{strange attractor} if there exists $z \in \mathcal{B}(A)$ such that:\\
    \begin{enumerate}
        \item $\overline{{\rm Orb}(z)} = A$;\\
        \item $\displaystyle \liminf_{n \to \infty} \frac{1}{n} \ln \|DF^n(z)\| > 0$;\\
        \item ${\rm int}(\mathcal{B}(A)) \neq \emptyset$.\\
    \end{enumerate}
\end{enumerate}
\end{definition}

\begin{definition}[Adapted from \cite{MO2015}]
We say that $F : M \to M$ (as above) possesses a \emph{strange attractor supporting an ergodic Sinai--Ruelle--Bowen (SRB) measure} $\nu$ if:\\
\begin{enumerate}
    \item $F$ exhibits an irreducible\footnote{It cannot be decomposed as the union of two or more disjoint attractors.} strange attractor
    \[
    \Omega = \bigcap_{m=0}^{\infty} F^m(\overline{U}),
    \]
    in an $F$-invariant open set $U \subset M$.\\
    \item The conditional measures of $\nu$ on unstable manifolds are absolutely continuous with respect to the Riemannian volume on those leaves.\\
    \item For Lebesgue almost every $z \in U$, the time average of any continuous function
    $\Phi : U \to \mathbb{R}$ converges to the spatial average:
    \[
    \lim_{n \to \infty} \frac{1}{n} \sum_{i=0}^{n-1} \Phi(F^i(z)) = \int \Phi \, d\nu.
    \]
\end{enumerate}
\end{definition}

The measure $\nu$ is physical in the sense that it describes the asymptotic distribution of orbits for a set of initial conditions of positive Lebesgue measure.
 We  say that a vector field possesses  a \emph{strange attractor supporting an ergodic  SRB measure} if the first return map on a cross section does. \\
 
 \begin{definition} [Adapted from \cite{BS93, BST98, Rodrigues2021}]
 \label{large_def}
Let $F: \mathcal{A} \to \mathcal{A}$ be map on a compact annulus\footnote{Also called by circloid.} $\mathcal{A} = I \times \EU^1$, where $I\subset \RR^+$ is a compact interval. A strange attractor $\Lambda \subset \mathcal{A}$ is said to be \emph{``large''} if it is not contained in any contractible subset of $\mathcal{A}$. Specifically, $\Lambda$ satisfies:\\
\begin{enumerate}
    \item it is non-contractible within the annulus, meaning it wraps around the central hole of $\mathcal{A}$ in the angular direction. \\
    \item  its dynamics are characterised by strong expansion in the $\EU^1$ direction and strong contraction in the $I$ direction, such that the image $F(\mathcal{A})$ is a ``thin" set that stretches across the full angular extent of the annulus.
\end{enumerate}
\end{definition}

\section{Setting}
\label{sec:setting}

Our object of study is the dynamics around the ``ghost'' of a homoclinic network associated with a saddle-focus undergoing a supercritical Hopf bifurcation, for which we provide a rigorous description here. \\

\subsection{Object of study}
\label{subsec:object_of_study}

We consider a three-parameter family of $C^3$-vector fields $f_{(\varepsilon,b_+, b_-)}: \mathbb{R}^3 \to \mathbb{R}^3$ defined by the Initial Value Problem:
\begin{equation}
\label{eq:ivp2}
\begin{cases}
\dot{x} = f_{(\varepsilon,b_+, b_-)}(x) \\
x(0) = x_0 \in \mathbb{R}^3
\end{cases}
\end{equation}
where $\varepsilon, b_+, b_- \in (-\tau, \tau)$ are independent parameters and $\tau > 0$ is arbitrarily small. The flow, denoted by $\Phi(t, x)$ for $x \in \mathbb{R}^3$ and $t \in \mathbb{R}$, satisfies the following properties when $b_+=b_-= 0$ (see Figure \ref{Imagem 1}):\\

\begin{itemize}
    \item[\textbf{(P1)}] For $\varepsilon < 0$, there exists a double homoclinic connection associated with a saddle-focus equilibrium $O=(0,0,0)$.  The linear part of $f_{(\varepsilon,0)}$ at $O$,   $Df_{(\varepsilon,0)}(O)$, possesses eigenvalues $\lambda$ and $\varepsilon \pm i \omega$, with $\lambda, \omega \in \mathbb{R}^+$. \\  
    
    \item[\textbf{(P2)}] At $\varepsilon=0$, the equilibrium $O$ undergoes a supercritical Hopf bifurcation, giving rise to a stable periodic solution $C_\varepsilon$ for $\varepsilon > 0$, admitting $\alpha>0$ as the \emph{first Lyapunov exponent} of the normal form.  \\
    
    \item[\textbf{(P3)}] For $\varepsilon > 0$, the equilibrium $O$ is an unstable equilibrium, the periodic orbit $C_\varepsilon$ is of saddle type, and both branches of its unstable manifold $W^u(C_\varepsilon)$ are reinjected near a neighbourhood of its local stable manifold $W^s(C_\varepsilon)$.\\
\end{itemize}

The specific geometry of the reinjection refereed in \textbf{(P3)} is detailed in Hypothesis~\textbf{(P4)}  of Section~\ref{ss:global}, where the role of the  parameters $b_+, b_-$ is described. The unfolding dynamics associated to this class of vector fields is what we call \emph{Shilnikov-Hopf bifurcations}.
 \\

\begin{figure}
\begin{center}
 \includegraphics[height=18.5cm]{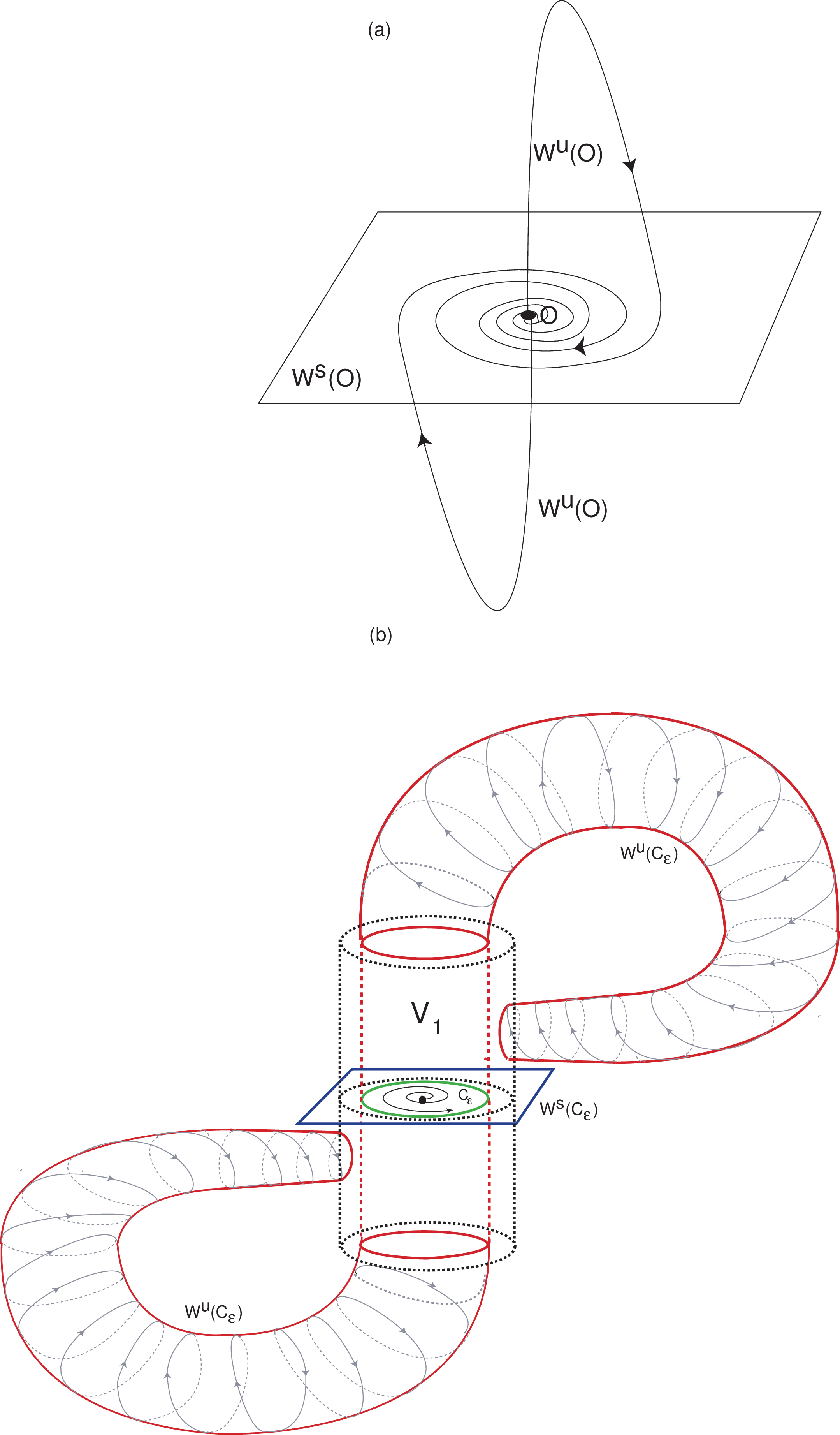}
\end{center}
\caption{\small (a) Illustration of Property  \textbf{(P1)} where $\varepsilon<0$ and $b_+=b_-=0$. (b) Illustration of Properties  \textbf{(P2)--(P3)--(P4)} for $\varepsilon>0$, $b_+>0$ and $b_-<0$.   }
 \label{Imagem 1}
\end{figure}

\section{Local and global maps}
\label{s: local map}
From now on, the symbol $\mathcal{O}$ denotes the standard big-O Landau notation. 
For $\varepsilon \geq 0$, we describe the local dynamics in cylindrical coordinates $(r, \varphi, z)$. Near the equilibrium $O$,   the following system of differential equations  captures the essential features of a supercritical Hopf bifurcation associated to \textbf{(P1)--(P2)} (cf. \cite[Eq. (1)]{BS93}):
\begin{equation}
\label{local_flow}
\begin{cases}
\dot{r} = \varepsilon r - \alpha r^3 + \mathcal{O}(r^5) \\ 
\dot{\varphi} = \omega + \mathcal{O}(r^2) \\
\dot{z} = \lambda z + \mathcal{O}(z^2).
\end{cases}
\end{equation}
System \eqref{local_flow} constitutes a $C^k$--approximation ($k \geq 3$) of $ f_{(\varepsilon,b_+, b_-)}$ satisfying properties \textbf{(P1)--(P3)} within a cylindrical neighbourhood $V_\rho$ of the origin $O$, of radius $\rho$ and height $2\rho$ (cf.  \cite[Ch. 3]{kuznetsov1998elements}):
\begin{equation}
\label{cylinder1}
V_\rho = \{ (r, \varphi, z) : 0 \leq r \leq \rho, \, |z| \leq \rho \},
\end{equation}
where $\rho > 0$ is chosen to be sufficiently small. The parameter $\alpha > 0$ in \eqref{local_flow} guarantees the supercritical nature of the Hopf bifurcation described in  \textbf{(P2)}. \\

\begin{figure}
\begin{center}
 \includegraphics[height=9.0cm]{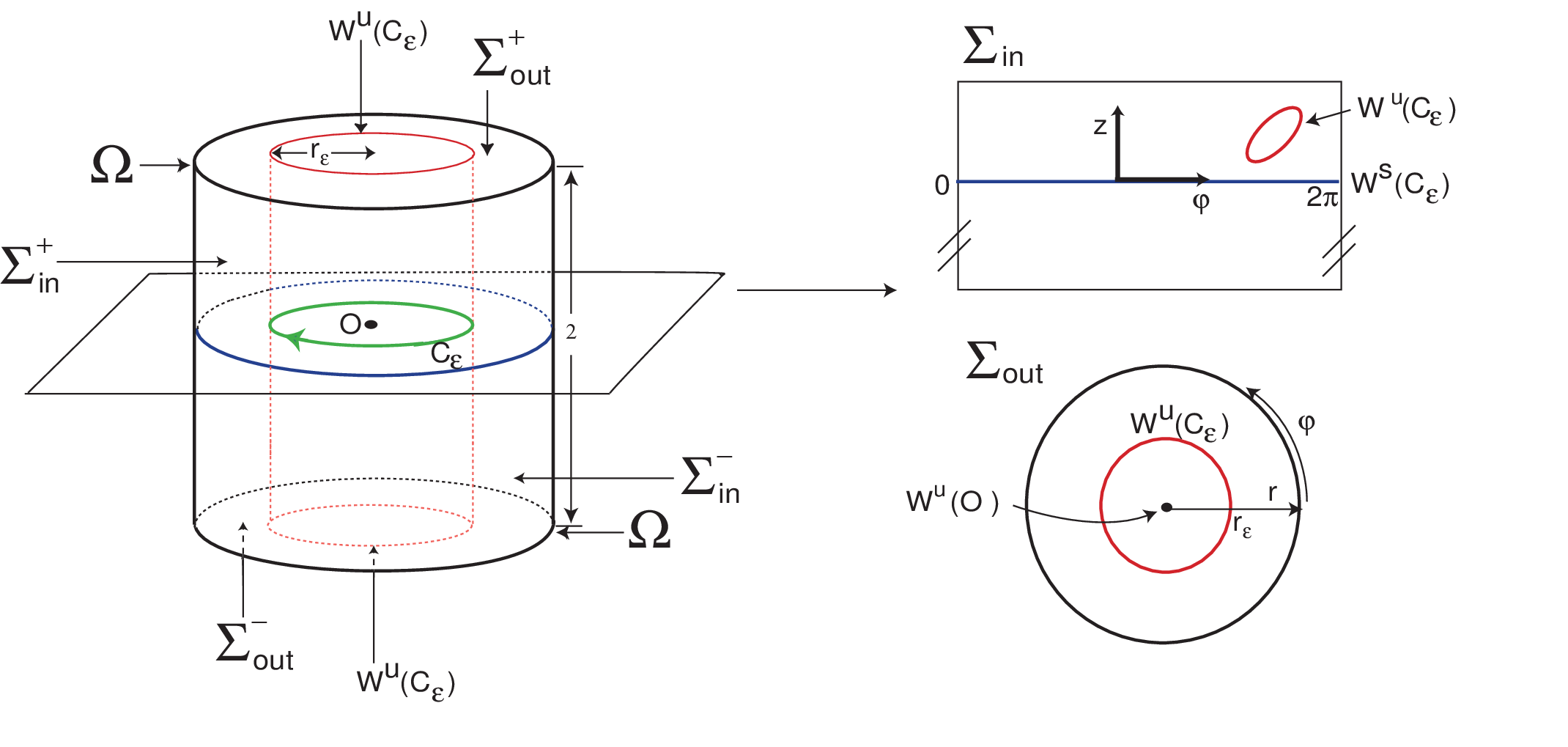}
\end{center}
\caption{\small  Local dynamics in the neighbourhood $V_1$ (see \eqref{cylinder1}). A trajectory entering through $\Sigma_{in}$ spirals near  $C_\varepsilon$,   exiting through one of the disks of   $\Sigma_{out}$. The circles $\Omega^\pm$ denote the intersection of the cylinder walls and the top/bottom disks. Double bars mean that the sides are identified. }
 \label{Imagem 2}
\end{figure}

By rescaling the variable $t$ and the coordinates $r$ and $z$, we may assume, without loss of generality, that $\omega = \rho = 1$. For $\varepsilon > 0$, the origin $O$ acts as a repeller within the $z=0$ plane, and the periodic solution $C_{\varepsilon} \subset V_1$ is given in local cylindrical coordinates by:
\begin{equation}
\label{epsilon<alpha}
z = 0, \qquad r_\varepsilon := \sqrt{\frac{\varepsilon}{\alpha}} < 1. 
\end{equation}
In the following local analysis, we disregard higher-order terms to focus on the leading-order dynamics. Note that at the bifurcation point $\varepsilon=0$, the equilibrium $O$ is non-hyperbolic.\\

\subsection{Cross sections}
\label{ss:cross-sections}

To study the transition map near  the periodic solution $C_\varepsilon$, we define the boundary of the neighbourhood \eqref{cylinder1} as $\partial V_1 = \Sigma_{in} \cup \Sigma_{out} \cup \Omega$, partitioned as follows (see Figure \ref{Imagem 2}):\\
\begin{itemize}
    \item $\Sigma_{in}$ is the   cylinder wall defined by $r=1$. This section is divided by the local stable manifold $W^s_{\text{loc}}(O)$ (when $\varepsilon \leq 0$) or $W^s_{\text{loc}}(C_\varepsilon)$ (when $\varepsilon > 0$). Trajectories originating at $\Sigma_{in}$ enter the interior of $V_1$ for $t > 0$.\\
    \item $\Sigma_{out}$ consists of the union of two disks (top and bottom) defined by $z = \pm 1$ and $r < 1$. For $\varepsilon > 0$, each of these disks is separated by the local unstable manifold $W^u_{\text{loc}}(C_\varepsilon)$. Trajectories starting at $\Sigma_{out}$ enter $V_1$ for $t < 0$, meaning they are the exit points for the forward flow.\\
    \item The vector field is strictly transverse to $\partial V_1$, except at the boundary circles $\Omega = \overline{\Sigma}_{in} \cap \overline{\Sigma}_{out}$, where the flow may exhibit tangencies to the cylinder boundaries.\\
\end{itemize}

For $\varepsilon > 0$, the entrance section $\Sigma_{in}^+$ is parameterised by $(\varphi, z) \in \mathbb{S}^1 \times (0, 1)$ and $\Sigma_{in}^-$  by $(\varphi, z) \in \mathbb{S}^1 \times (-1, 0)$. The intersection $\Sigma_{in} \cap W^s_{\text{loc}}(C_\varepsilon)$ corresponds to the circle at $z=0$. \\

As depicted in Figure \ref{Imagem 2}, the exit sections $\Sigma_{out}^\pm$ are parameterised by $(\varphi, r) \in \mathbb{S}^1 \times [0, 1)$ with $z = \pm 1$. Crucially, the set $\Sigma_{out} \cap W^u_{\text{loc}}(C_\varepsilon)$ consists of two circles of radius $r = r_\varepsilon$ located on the disks $z=1$ and $z=-1$, respectively. \\

\subsection{The local transition map}
\label{sublocal}

For a trajectory starting at the entrance section, with initial condition $(\varphi, z) \in \Sigma_{in} \setminus \{z=0\}$, the time required to reach the exit disks $\Sigma_{out}^\pm$ is governed by the  growth in the $z$-direction. This \emph{time of flight} is given by:
\begin{equation}
 \label{Time of Flight}
 \tau (\varphi, z) = \frac{1}{\lambda} \ln \left( \frac{1}{|z|} \right).
\end{equation}
During this interval, the radial component evolves according to the separable differential equation $\dot{r} = \varepsilon r - \alpha r^3$. Assuming $r(0) \neq 0$, this  equation admits the explicit solution:
\begin{equation*}
r(t) = \left[ \frac{\alpha}{\varepsilon} + \left( r(0)^{-2} - \frac{\alpha}{\varepsilon} \right) e^{-2\varepsilon t} \right]^{-1/2}.
\end{equation*}

By setting the initial radial condition $r(0)=1$ (consistent with the boundary $\Sigma_{in}$) and substituting  $t= \tau(\varphi, z)$, we derive the local map $\text{Loc}: \Sigma_{in} \setminus \{z=0\} \to \Sigma_{out}$, which characterises the local transition:
\begin{equation}
\label{loc1}
\text{Loc}(\varphi, z) = \left( \left[ \frac{\alpha}{\varepsilon} + \left( 1 - \frac{\alpha}{\varepsilon} \right) |z|^{\frac{2\varepsilon}{\lambda}} \right]^{-1/2}, \varphi - \frac{1}{\lambda} \ln |z| \right) =: (r_1, \varphi_1).
\end{equation}
\bigbreak 

\begin{remark}[Continuity at the bifurcation point $\varepsilon=0$]
The radial component $r_1$ is continuous with respect to the parameter $\varepsilon$ at $0$. This can be verified by expressing the term $|z|^{\frac{2\varepsilon}{\lambda}}$ as an exponential:
\begin{eqnarray*}
r_1(\varphi, z) &=& \left[ \frac{\alpha}{\varepsilon} + \left( 1 - \frac{\alpha}{\varepsilon} \right) \exp \left( \frac{2\varepsilon}{\lambda} \ln |z| \right) \right]^{-1/2} \\
&\overset{\text{Taylor}}{\approx}& \left[ 1 - \frac{2\alpha}{\lambda} \ln |z| + \varepsilon \left( \frac{2}{\lambda} \ln |z| - \frac{2\alpha}{\lambda^2} (\ln |z|)^2 \right) \right]^{-1/2} + \mathcal{O}(\varepsilon^2).
\end{eqnarray*}
In the limit $\varepsilon \to 0^+$, we recover the non-hyperbolic radial map $r_1 = (1 - \frac{2\alpha}{\lambda} \ln |z|)^{-1/2}$, which is in full agreement with the local results for the Shilnikov-Hopf scenario described in \cite[pp. 221]{BS93}.
\end{remark}
\bigbreak
\subsection{Local dynamics and domain geometry}

For $\varepsilon<\alpha$ and $\varrho >0$ small, let $\mathcal{A}^\pm$ be a semi-closed annulus in $\Sigma_{out}^\pm$ defined by:\\
\begin{equation}
\label{A_delta}
\mathcal{A}^\pm = \left\{(r_1, \varphi_1) \in \Sigma_{out}^\pm : r_1 \in (r_\varepsilon, r_\varepsilon + \varrho] \quad \text{and} \quad \varphi_1 \in \mathbb{S}^1 \right\}.
\end{equation}
\bigbreak
The geometry of the first return map's domain is established by the following technical results. The proof of  the first is immediate by analysing the expressions \eqref{loc1} and \eqref{A_delta}.\\

\begin{lemma}
\label{F_pm}
The topological closure of the set $\mathcal{F}^\pm = \text{Loc}^{-1}(\mathcal{A}^\pm)$ corresponds to a horizontal strip in $\Sigma_{in}^\pm \cup W^s_{\text{loc}}(C_\varepsilon)$, parameterised as:
\begin{equation*}
\left\{(\varphi, z) \in \Sigma_{in} : \varphi \in \mathbb{S}^1 \quad \text{and} \quad z \in [0, \pm M] \right\},
\end{equation*}
where the maximum height $M$ is given by:
\begin{equation*}
M = \left[ \frac{(r_\varepsilon + \varrho)^{-2} - \frac{\alpha}{\varepsilon}}{1 - \frac{\alpha}{\varepsilon}} \right]^{\frac{\lambda}{2\varepsilon}} \in (0,1).
\end{equation*}
\end{lemma}

\bigbreak
We now introduce the notion of \emph{segment} and \emph{spiral accumulation on a curve}. These concepts will be useful in the sequel. \\

\begin{definition}
A \emph{vertical segment} in $\Sigma_{in}^\pm$ is a smooth regular parametrised curve $\beta: [0,1] \to \Sigma_{in}^\pm$ that meets $W^s_{\loc}(C_\varepsilon)$ transversely at the point $\beta(0)$ only, such that $\beta(s) = (\varphi_0, z(s))$, where $\varphi_0 \in \mathbb{S}^1$ and $z(s)$ is a monotonic function of $s$ with $z(0)=0$.\\
\end{definition}

\begin{definition}[\cite{ALR}, adapted]
\label{spiralling1_def}
Let $a,b \in \mathbb{R}^+$ such that $a < b$ and let $H$ be an annulus parametrised by the coordinates $(h, \theta) \in [a,b] \times \mathbb{S}^1$. As depicted in Figure \ref{spiral1}, a \emph{spiral on $H$ accumulating on the circle} $\mathcal{C} \subset H$ (where $\mathcal{C}$ is the circle $h = h_0$ for some $h_0 \in [a,b]$) is a smooth curve 
$\alpha : [c, +\infty[ \to H$ such that, in coordinates $\alpha(s) = (h(s), \theta(s))$:\\
\begin{enumerate}
    \item $\dpt \lim_{s \to +\infty} h(s) = h_0$;\\
    \item $\dpt \lim_{s \to +\infty} |\theta(s)| = +\infty$ (monotonic rotation);\\
    \item The curve is simple (no self-intersections).\\
\end{enumerate}
\end{definition}

\begin{figure}[ht]
\begin{center}
\includegraphics[height=4.1cm]{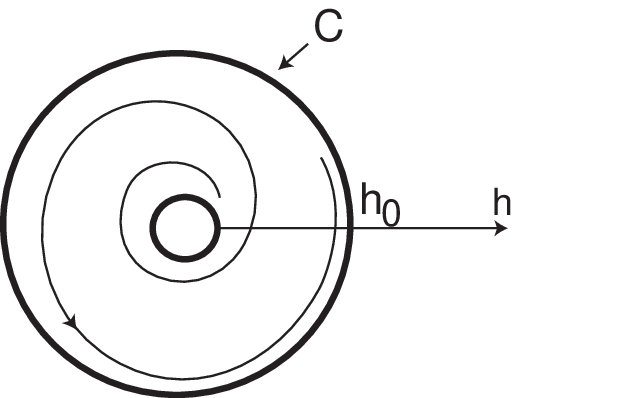}
\end{center}
\caption{\small Spiral on the annulus accumulating on the circle $\mathcal{C}$. The arrow indicates the direction of evolution of the variable $s$ from Definition \ref{spiralling1_def}.}
\label{spiral1}
\end{figure}

\begin{proposition} 
The following assertions hold for $\varphi_0,\varphi_0^\star \in \mathbb{S}^1$ and $z_0 \in (-1,1)\setminus\{0\}$:\\
\begin{enumerate}
    \item The radial map $r_1$ in \eqref{loc1} is strictly increasing with respect to $|z|$.\\
    \item The function $\text{Loc}$ maps horizontal arcs in $\Sigma_{in}$ defined by 
    $$\{(\varphi, z)\in \mathbb{S}^1 \times [-1,1]\setminus \{0\} : \varphi \in [\varphi_0 ,\varphi_0+ \varphi_0^\star ], \, z = z_0\}$$ 
    into arcs of a circle in $\Sigma_{out}$ with angular length $ \varphi_0^\star$.\\
    \item The function $\text{Loc}$ maps vertical segments $\beta(s)= ( \varphi_0, z(s))$ into spirals in $\Sigma_{out}$ accumulating on the circle $r = r_\varepsilon$.\\
\end{enumerate}
\end{proposition}

\begin{proof}
Consider the Local map $\text{Loc}: \Sigma_{in} \to \Sigma_{out}$ given by the coordinates $(r_1, \varphi_1)$ as defined in \eqref{loc1}, where:
\begin{equation}
\label{eq_map_formal_final}
r_1(z) = \left[ \frac{\alpha}{\varepsilon} + \left( 1 - \frac{\alpha}{\varepsilon} \right) |z|^{\frac{2\varepsilon}{\lambda}} \right]^{-1/2}, \quad \varphi_1(\varphi, z) = \varphi - \frac{1}{\lambda} \ln |z|.
\end{equation}

\begin{enumerate}
    \item 
    Let $u(|z|) = \frac{\alpha}{\varepsilon} + (1 - \frac{\alpha}{\varepsilon}) |z|^{\frac{2\varepsilon}{\lambda}}$. Then $r_1(|z|) = [u(|z|)]^{-1/2}$. Differentiating $r_1$ with respect to $|z|\neq 0$ via the chain rule, we may write:\\
    \begin{equation*}
        \frac{d r_1}{d |z|}|z|= -\frac{1}{2} [u(|z|)]^{-3/2} \cdot \left( 1 - \frac{\alpha}{\varepsilon} \right) \frac{2\varepsilon}{\lambda} |z|^{\frac{2\varepsilon}{\lambda}-1} = \frac{\varepsilon}{\lambda} |z|^{\frac{2\varepsilon}{\lambda}-1} [u(|z|)]^{-3/2} \left( \frac{\alpha}{\varepsilon} - 1 \right). \\
    \end{equation*}
    \bigbreak
    In the supercritical case, the parameters satisfy $\alpha > \varepsilon$ and $\lambda, \varepsilon > 0$. Since $(\frac{\alpha}{\varepsilon} - 1) > 0$, it follows that $\frac{d r_1}{d |z|} |z|> 0$ for all $|z| \in (0, 1]$. Thus, $r_1$ is strictly increasing with respect to $|z|$.\\

    \item 
    For a fixed $z = z_0$, $r_1(z_0)$ is constant (denoted $R_0$). The angular component becomes $\varphi_1(\varphi) = \varphi + C$, where $C = -\frac{1}{\lambda} \ln |z_0|$. The image of the interval $\varphi \in [\varphi_0 , \varphi_0  +\varphi_0^\star]$ is the set $\{ (R_0, \varphi_1) : \varphi_1 \in [\varphi_0  + C, \varphi_0  + C + \varphi_0^\star] \}$, which is an arc of a circle of radius $R_0$ with angular length $\varphi_0^\star$.\\

    \item \
    Let $\beta(s) = (\varphi_0, z(s))$ be a vertical segment such that $z(s) \to 0$ as $s \to 0^+$. Let $\alpha(s) = \text{Loc}(\beta(s)) = (h(s), \theta(s))$. 
       We now verify the conditions of Definition \ref{spiralling1_def} for accumulation on $h_0 = r_\varepsilon = \sqrt{\varepsilon/\alpha}$.\\
    \begin{itemize}
        \item Since $z(s)$ is monotonic, then $\theta(s) = \varphi_0 - \frac{1}{\lambda} \ln |z(s)|$ is monotonic. Furthermore, $$\lim_{s \to 0^+} |\theta(s)| = \lim_{z \to 0^+} \left| \varphi_0 - \frac{1}{\lambda} \ln |z| \right| = +\infty.$$\\
        \item The coordinate $h(s) = r_1(z(s))$ satisfies:
        \begin{equation*}
        \lim_{s \to 0^+} h(s) = \left[ \frac{\alpha}{\varepsilon} + \lim_{z \to 0^+} \left( 1 - \frac{\alpha}{\varepsilon} \right) |z|^{\frac{2\varepsilon}{\lambda}} \right]^{-1/2} = \sqrt{\frac{\varepsilon}{\alpha}} = r_\varepsilon.
        \end{equation*}
        Since $r_1$ is strictly increasing and $z(s)$ is monotonic, $h(s)$ converges monotonically to $r_\varepsilon$. This confirms that the image of $\beta$ is a spiral accumulating on the circle $r = r_\varepsilon$.
    \end{itemize}
   
\end{enumerate}
\end{proof}

\subsection{Global map}
\label{ss:global}

For $\varepsilon \leq 0$ and $b_+=b_-=0$, the Flow-box Theorem \cite{PM, PT, Smale} ensures that there exists $\varrho>0$ for which the annulus $\mathcal{A}^\pm$ around $W^u_{\text{loc}}(C_\varepsilon) \cap \Sigma_{out}^\pm$ is mapped into a neighbourhood of $W^s_{\text{loc}}(C_\varepsilon) \cap \Sigma_{in}$. This is a consequence of the existence of the double homoclinic connection to $O$ established in \textbf{(P1)}. This implicitly defines a \emph{global map} $\Psi_\pm : \Sigma_{out}|_{\mathcal{A}^\pm} \to \Sigma_{in}$. \\

In $\Sigma_{out}^\pm$, we define Cartesian coordinates as follows:
\begin{equation*}
(x_1, y_1) = (r_1 \cos \varphi_1, r_1 \sin \varphi_1),
\end{equation*}
where $(r_1, \varphi_1)$ are the local polar coordinates of $\Sigma_{out}^\pm$ defined in Subsection \ref{ss:cross-sections}. Let $(\varphi_2, z_2)$ denote the coordinates of the image $\Psi_\pm(x_1, y_1)$. To formalize our results, we establish the following technical assumption which locate \eqref{eq:ivp2} into the Configuration (a) of \cite[pp. 226]{BS93}:\\

\begin{itemize}
    \item[\textbf{(P4)}] There exist $b_+ \in (m_+ r_\varepsilon,  1-m_+ r_\varepsilon )$, $b_- \in (-1+m_- r_\varepsilon, -m_- r_\varepsilon)$ and $\Delta \varphi \in \EU^1$ such that the map $\Psi_+$ can be written in a neighbourhood of $(0,0)_+ \in \Sigma_{out}^+$ as:
    \begin{equation*}
    \begin{pmatrix} \varphi_2 \\ z_2 \end{pmatrix} = \Psi_+ \left(\begin{pmatrix} x_1 \\ y_1 \end{pmatrix}_+ \right)\mapsto \begin{pmatrix} 0 \\ b_+ \end{pmatrix} + r_1 A_+ \begin{pmatrix} \cos \varphi_1 \\ \sin \varphi_1 \end{pmatrix}.
    \end{equation*}
    The map $\Psi_-$ is written in a neighborhood of $(0,0)_- \in \Sigma_{out}^-$ as:
    \begin{equation*}
    \begin{pmatrix} \varphi_2 \\ z_2 \end{pmatrix} = \Psi_- \left( \begin{pmatrix} x_1 \\ y_1 \end{pmatrix}_-\right) \mapsto \begin{pmatrix} \Delta \varphi \\ b_- \end{pmatrix} + r_1 A_- \begin{pmatrix} \cos \varphi_1 \\ \sin \varphi_1 \end{pmatrix},
    \end{equation*}
    where $A_\pm = \begin{pmatrix} a_{11, \pm} & a_{12, \pm} \\ a_{21, \pm} & a_{22, \pm} \end{pmatrix}$ is an invertible matrix satisfying $|a_{ij, \pm}| < 1$ for all $i, j \in \{1,2\}$. \\ \bigbreak
    \end{itemize}
See Figure \ref{Imagem 6}. From now on, let use the following terminology $m_\pm = \sqrt{a_{21, \pm}^2 + a_{22, \pm}^2} $.

\bigbreak

\begin{figure}
\begin{center}
 \includegraphics[height=14.5cm]{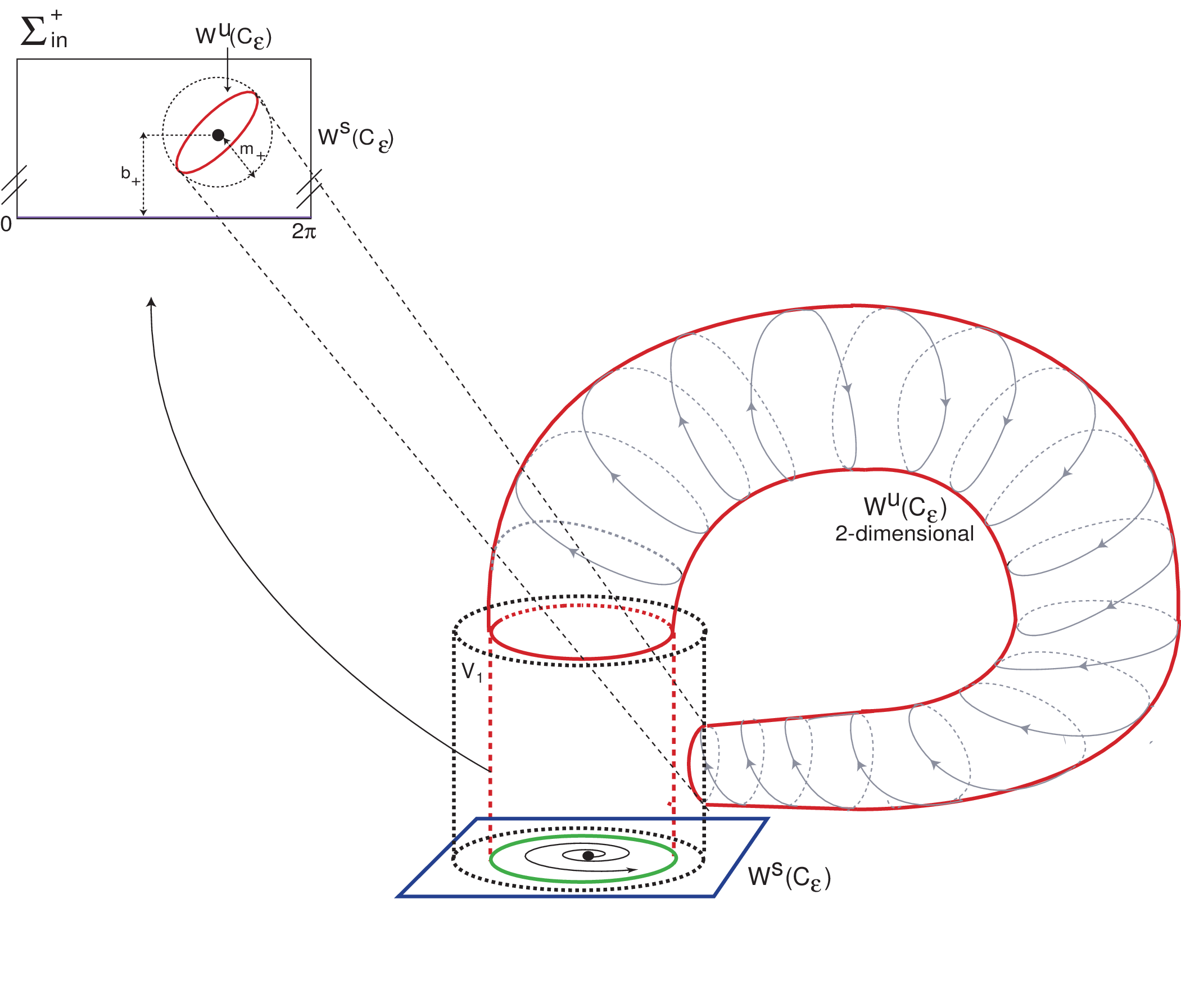}
\end{center}
\caption{\small  Meaning of $b\equiv b_+>0$ and $m\equiv m_+ >0$. This corresponds to the Configuration (a) of \cite[pp. 226]{BS93}. }
 \label{Imagem 6}
\end{figure}

\begin{remark}[Digestive remarks]
\label{dig_remarks1}
We highlight several key remarks essential for understanding the dynamical configuration under analysis: \\
\begin{enumerate}
\item  Property \textbf{(P4)} describes the ``long trip" (global map) that the trajectories take outside  $V_1$.  It   tells us where the unstable manifold ``lands" (at heights $b_\pm$) and how the nearby points are  scaled by the matrices $A_\pm$ before starting a new cycle -- see Figure \ref{Imagem 6}.    \\

\item For $\varepsilon < 0$, $b_\pm$ ``measures'' the distance between $\Psi_\pm(W^u_{\text{loc}}(O))$ and $W^s_{\text{loc}}(O)$ in $\Sigma_{in}$ as in the  Shilnikov-like transition  \cite{hom, Shilnikov65, Shilnikov67A, Shilnikov70}. For $\varepsilon<0$, one knows that $|\varepsilon|<\lambda$ , the condition needed to have suspended horseshoes close to the cycle for $b_\pm$ small -- cf. \cite{Tresser}. \\

    \item  It is not restrictive to assume that the $\varphi_2$ component of $\Psi_+((0,0)_+)$ is zero; this is achieved by a simple rotation of the coordinate system.  It follows from \textbf{(P4)} that $\Psi_\pm$ is smooth and satisfies $$\Psi_+((0,0)_+) = (0, b_+) \quad \text{and} \quad \Psi_-((0,0)_-) = (\Delta \varphi, b_-). \\$$ 
  \bigbreak  
    \item         Property \textbf{(P4)} implies that $0\neq m_\pm \ll1$, meaning that there is a contraction along the homoclinic connection. This property is not unnatural since it may happens in the context of Celestial Mechanics  (see \cite[pp. 222]{BS93} and references therein).\\

    \item  If $|b_\pm| + m_\pm r_\varepsilon < 1$, there exists $\varrho>0$ such that the image of $\mathcal{A}^\pm \subset \Sigma_{out}^\pm$ under the global map $\Psi_\pm$ is a small elliptical annulus in $\Sigma_{in}$ (remind the the meaning of  $\mathcal{A}^\pm$   in Lemma \ref{F_pm}). Specifically, $\Psi_\pm$ maps the circle $\Sigma_{out} \cap W^u_\loc(C_\varepsilon)$ into a contractible curve in $\Sigma_{in}$ -- cf. \cite[Fig. 6]{HK1993}.  In this case,  $$ \text{Loc} \circ \Psi_\pm (\mathcal{A}^\pm ) \subset \mathcal{A}^+\cup \mathcal{A}^- ,$$ provided the domain is appropriately restricted (cf.  \cite[Prop. 2.1]{BS93}).\\
    
    \item  If $b_+ \in (m_+r_\varepsilon,1-m_+r_\varepsilon)$  and  if the contraction of $\Psi_+$ is sufficiently strong ($0\ll m_+<\lambda$), then the map $\text{Loc} \circ \Psi_+|_{\mathcal{A}^+}$ is well-defined on $\mathcal{A}^+$. In particular,  $\text{Loc} \circ \Psi_+(\mathcal{A}^+) \subset \mathcal{A}^+$. An analogous result holds for $\mathcal{A}^-$. This scenario is the same of \cite[Fig. 1, $\sigma^2>B^2 \mu$]{HK1993}. In this case,   trajectories do not crash directly into the homoclinic connection to $C_\varepsilon$, but stay close enough to be affected by its strong stretching, as suggested by Figure \ref{Imagem 3}. \\
    
    \item  If  $|b_\pm| \leq  m_\pm r_\varepsilon $ $(\Leftrightarrow$ Property \textbf{(P4)} does not hold), homoclinic orbits to $C_\varepsilon$ emerge and then homoclinic tangles (suspended Smale horseshoes) occur in the neighbourhood of the homoclinic cycles (cf.  \cite[pp. 222]{BS93}). In the present paper, we are not interested in this case.\\
\end{enumerate}
\end{remark}

\subsection{Focus of the present paper and notation}
From now on, we are going to concentrate the attention on the Configuration (a) of \cite[pp. 226]{BS93} illustrated in Figure \ref{Imagem 6}:
 $$b_+ \in (0,  1-m_+ r_\varepsilon)\subset \RR^+ \quad \text{and} \quad b_- \in (-1+m_- r_\varepsilon, 0 )\subset \RR^-. $$ In this case, the first return map to $\mathcal{A}_\pm$ may be studied looking separately at two return maps and the main results of Section \ref{s: first return} are valid in both cases.   For the sake of simplicity, we are going to focus our study in the first case ($Loc\circ \Psi_+$) and drop the $\pm$ indices, setting:
 $$
 m=m_+,\, \qquad b= b_+>0, \qquad \Sigma_{out}=\Sigma_{out}^+,  \qquad \mathcal{A}=\mathcal{A}^+\qquad\text{and}\qquad \Psi=\Psi_+.
 $$
 Instead of \eqref{eq:ivp2}, we consider hereafter the codimension two problem:
 \begin{equation}
\label{eq:ivp}
\begin{cases}
\dot{x} = f_{(\varepsilon,b)}(x) \\
x(0) = x_0 \in \mathbb{R}^3
\end{cases}
\end{equation}
where    $f_{(\varepsilon,b)}: \mathbb{R}^3 \to \mathbb{R}^3$ is a $C^3$--vector field,   $\varepsilon, b  \in (-\tau, \tau)$ are independent, $\tau > 0$ is arbitrarily small and its flow satisfies \textbf{(P1)--(P4)}, under appropriate adaptations. Let us denote by $\mathcal{X}_{SH}$ the set of vector fields of \eqref{eq:ivp} whose flow satisfies \textbf{(P1)--(P4)}, endowed with the $C^3$--Whitney topology.

 \begin{figure}
\begin{center}
 \includegraphics[height=12.5cm]{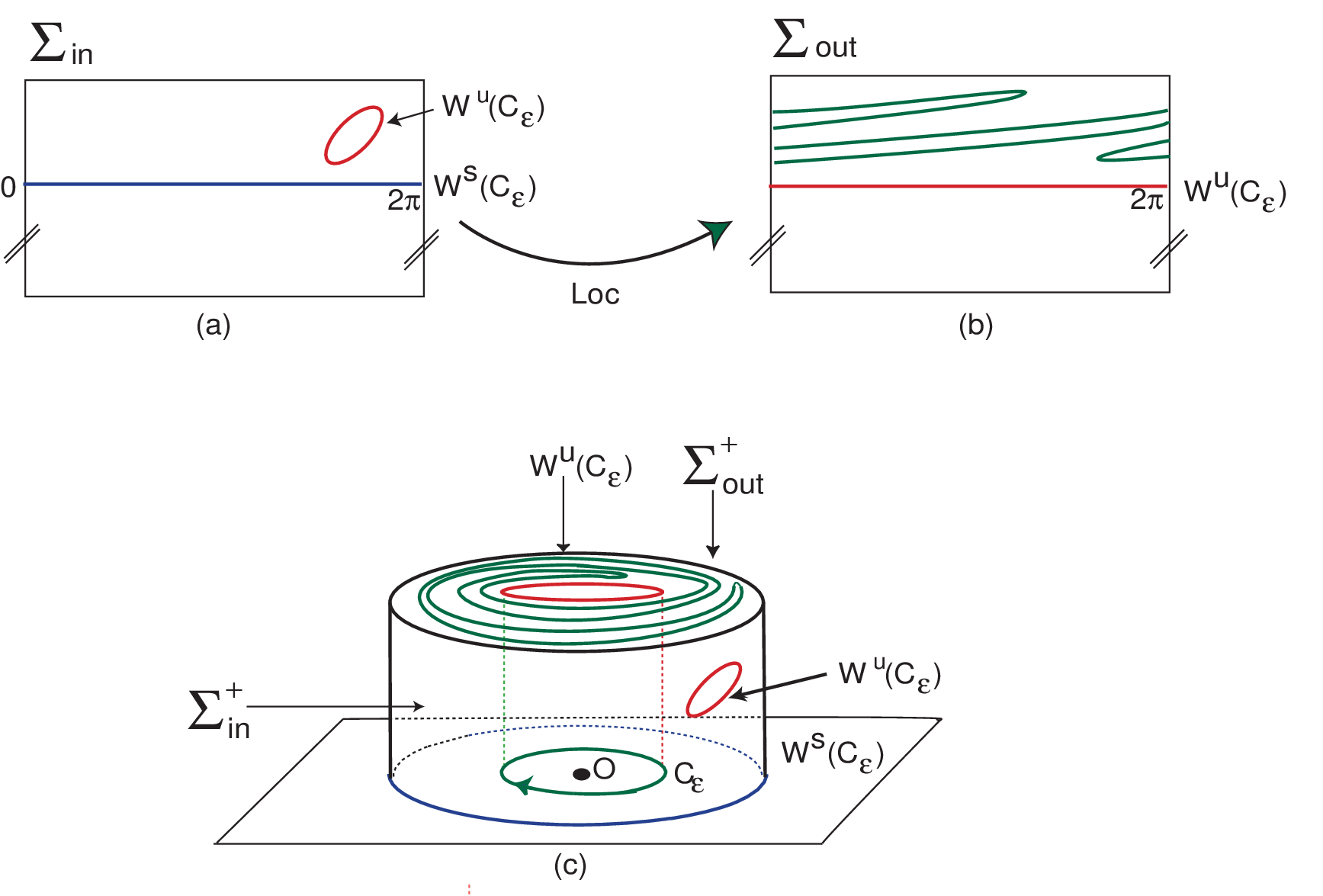}
\end{center}
\caption{\small   (a) Illustration of $\Sigma_{in}$; the circle represents $\Psi_+(W^u_\loc(C_\varepsilon)\cap \Sigma_{out}^+)$. (b)   Illustration of $\mathcal{A}^\pm$ in one connected component of $\Sigma_{out}$;  the spiralling line corresponds to $Loc\circ \Psi_+(W^u_\loc(C_\varepsilon)\cap \Sigma_{out}^+)$. (c) View of the structures of (a) and (b) in $V_1$. }
 \label{Imagem 3}
\end{figure}

 \section{The first return map}
\label{s: first return}
From Hypothesis \textbf{(P1)}, it is known that for $\varepsilon \leq 0$ and $b =0$, there exists a   homoclinic connection to $O$. Consequently, recurrent dynamics exist in a neighbourhood of the homoclinic network, allowing the definition of a first return map on a subset of $\Sigma_{out}$.  \\

\begin{proposition}
\label{thm:0}
Let  $f_{(\varepsilon,b)} \in \mathcal{X}_{SH}$  with $0<\varepsilon <\alpha$ \eqref{epsilon<alpha}. Then there exists $\delta > 0$ such that the first return map $\mathcal{R}_b$ associated $f_{(\varepsilon,b)}$,  acting on the cross-section $\mathcal{A}\subset \Sigma_{out}$ (see \eqref{A_delta}), may be $C^3$--approximated  by: \\
\begin{equation}
\label{def_Rb}
\mathcal{R}_b(r_1, \varphi_1) :=
\begin{pmatrix} 
G_b(r_1, \varphi_1) \\ 
F_b(r_1, \varphi_1)  
\end{pmatrix} =
\begin{pmatrix} 
\left[ \frac{\alpha}{\varepsilon} + \left( 1 - \frac{\alpha}{\varepsilon} \right) (b + mr_1 \cos \varphi_1)^{\frac{2\varepsilon}{\lambda}} \right]^{-1/2} \\ 
 r_1 (a_{11}\cos \varphi_1 + a_{12}\sin \varphi_1) - \frac{1}{\lambda} \ln (b + m r_1  \cos \varphi_1) \pmod{2\pi}
\end{pmatrix}
\end{equation}
whose domain is given by:
\begin{equation*}
\mathcal{D} = \{ (r_1, \varphi_1) \in [r_\varepsilon, r_\varepsilon + \delta] \times \EU^1 : b + m r_1 \cos \varphi_1 > 0 \}\subset \mathcal{A}.
\end{equation*}
\end{proposition}

\bigbreak
\begin{proof}

If $b_+ \in (m_+r_\varepsilon,1-m_+r_\varepsilon)$, then there exists $\mathcal{A} \subset \Sigma_{out}$ such that the first return map $\mathcal{R}_b = \text{Loc} \circ \Psi$ is well-defined on  $\mathcal{A}  \subset \Sigma_{out} $ -- as already noticed   in Remark \eqref{dig_remarks1} --- if necessary take a smaller $\varrho$ value in the definition  \eqref{A_delta}. \\

From the expression of $\Psi$ in \textbf{(P4)}, the $y_1$-coordinate is mapped to $z_2 = b + r_1(a_{21}\cos \varphi_1 + a_{22}\sin \varphi_1)$. Using a trigonometric identity, there exists $\varphi_0\in \EU^1$ such that:
\begin{equation*}
a_{21}\cos \varphi_1 + a_{22}\sin \varphi_1 = m \cos(\varphi_1 + \varphi_0),
\end{equation*}
where $m = \sqrt{a_{21}^2 + a_{22}^2}$ and $\tan \varphi_0 = -a_{22}/a_{21}$ provided $a_{21}\neq 0$. By performing a rigid rotation of the angular coordinate in $\Sigma_{out}$, we may assume $\varphi_0 = 0$ without loss of generality. 
Substituting $z_2$ into the local map   derived in  \eqref{loc1}, we obtain:
\begin{equation*}
r_2= \left[ \frac{\alpha}{\varepsilon} + \left( 1 - \frac{\alpha}{\varepsilon} \right) (b + mr_1 \cos \varphi_1)^{\frac{2\varepsilon}{\lambda}} \right]^{-1/2}
\end{equation*}
and the angular component follows from the translation $\varphi_2 - \frac{1}{\lambda} \ln z_2$. This approximation holds in the $C^3$--topology because of \eqref{local_flow} and the embedding $\Psi$ of \textbf{(P4)}. The need of the hypothesis $0<\varepsilon <\alpha$ comes from \eqref{epsilon<alpha}.
  \\
\end{proof}
With respect to $\mathcal{R}_b$, the map  $G_b$ represents the \emph{radial} return, and $F_b$ the \emph{angular} component. It is easy to see that if \textbf{(P4)} holds, then $\mathcal{D}$ is a circloid. Assume that  
\begin{equation}
\label{def_D}
 \mathcal{D}= I\times \EU^1= [r_{min}, r_{max}]\times \EU^1, \quad \text{where} \quad r_{min}= r_\varepsilon<r_{max}.
 \end{equation}
 
Throughout this paper, the notation $m \ll \lambda \ll 1,$ means that there exist positive constants $\eta_0$ and $\lambda_0$ such that $0 < m \leq \eta_0 \lambda$ and $ 0 < \lambda < \lambda_0$ where $\eta_0>0$ is chosen sufficiently small and $\lambda_0>0$ is fixed. Whenever necessary, $\eta_0$ may be further reduced so that all estimates established below remain valid.

\section{Main results}
\label{s:main}

Following  Proposition~\ref{thm:0}, we now focus our attention on the first return map $\mathcal{R}_b$, defined on the domain $\mathcal{D} \subset \Sigma_{out}$.  The following results provide a precise mathematical formulation of the \emph{Bosch--Sim\'o Conjecture} stated in \cite[p.~228]{BS93}. Let $ {\emph Leb}$ denote the one-dimensional Lebesgue measure in $\RR$ and let $\varrho > 0$ be an arbitrarily small number. 
\begin{maintheorem}
\label{main1} 
For $\varepsilon < \alpha$ and $m \ll \lambda \ll 1$, then:
\begin{equation*}
\liminf_{\varrho\to0^+}
\frac{\rm Leb(\Delta_\varrho)}
{\varrho}
>0.
\end{equation*}
where
$$
\Delta_\varrho = \{ b \in [m r_\varepsilon, m r_\varepsilon +\varrho] : \mathcal{R}_b \text{ exhibits a ``large'' strange attractor } \Lambda_b \}.
$$
\bigbreak
Furthermore, if $b\in \Delta_\varrho $, then the strange attractor $\Lambda_b$ supports a unique ergodic SRB measure $\mu_b$.\\
\end{maintheorem}

 Theorem \ref{main1} establishes that the existence of ``large'' strange attractors is an \emph{abundant phenomenon}, in the sense of \cite{MV93}, as $b \to (m r_\varepsilon)^+$. The iterates of Lebesgue-almost in the ``ghost'' of the homoclinic cycle   wind around an annulus $\mathcal{D}$   in the phase space, which justifies the terminology \emph{``large'' strange attractor} of  Definition \ref{large_def}. The proof, presented in Section \ref{proof_main_1}, relies on reducing the two-dimensional dynamics of $\mathcal{R}_b$ to an effective one-dimensional map using the theory of \emph{rank-one attractors} (see Section \ref{s: theory}). The condition $m \ll \lambda \ll 1$ means that the radial contraction $m$ is sufficiently strong to dominate  $\lambda$ which, in turn, is taken to be sufficiently small. \\
 
One might intuitively expect $\mathcal{R}_b$ to possess a strange attractor for every $b\gtrsim m r_\varepsilon$. However, this is not the case, as highlighted by our next result concerning the coexistence of periodic sinks.\\

\begin{definition}
A stable hyperbolic periodic orbit of the map $\mathcal{R}_b$ is a periodic orbit whose Jacobian matrix  $D\mathcal{R}_b$ has all eigenvalues with moduli strictly less than $1$. Such an orbit is termed \emph{superstable} if at least one of these eigenvalues is zero. \\
\end{definition}

\begin{maintheorem}
\label{main2} 
For $0<\varepsilon < \alpha$ and $m \ll \lambda$, there exists a sequence of parameters $(b_n)_{n \ge 1}$ with $\dpt \lim_{n \to \infty} b_n = (m r_\varepsilon)^+$ such that, for each $n\in \NN$, the map $\mathcal{R}_{b_n}$ possesses a superstable 2-periodic orbit.\\
\end{maintheorem}

The existence of these superstable 2-periodic orbits is studied in Section \ref{proof_main_2}. These critical points are located in a neighbourhood of $0$ or $\pi$ for the reduced map, confirming the numerical observations in \cite{BS93}. \\

\subsection*{Insight into the Reasoning}
\label{sec:insight}

The proofs of Theorem \ref{main1} and Theorem \ref{main2} rely on the reduction of  \eqref{def_Rb} to a family of one-dimensional maps $f_b: \mathbb{S}^1 \to \mathbb{S}^1$, enabled by the strong radial contraction of the system. The core mechanism driving the dynamics is the logarithmic term in the angular component, which induces a singular behaviour as the parameter $b\to   (m r_\varepsilon)^+$.
The proof of Theorem \ref{main1} is rooted in the  Wang--Young theory for non-uniformly expanding maps. The reasoning follows a delicate balance between expansion and recurrence.

 The proof of Theorem \ref{main2} shifts focus from measure-theoretic expansion to the topological ``winding" of the map's images.
 As  $b\to   (m r_\varepsilon)^+$, the velocities of the critical images diverge. We show that the relative distance between the images of the two \emph{critical points} winds around the circle infinitely many times. 
 Because the radial direction is always contracting, then $f_b$ automatically creates a superstable orbit.

\section{Existence of an invariant curve for $G_b$}
\label{s: invariant curve}

The purpose of this section is to justify the reduction of the
two-dimensional return map $\mathcal R_b$ to an effective
one-dimensional dynamics.  The key mechanism is the strong contraction in the radial direction.
Throughout this section we work on a compact trapping region
$D=I\times\mathbb S^1,$
where $I=[r_{\min},r_{\max}],$
and assume that
$\mathcal R_b(D)\subset D.
$
Moreover, we restrict attention to compact parameter intervals
$B=[b_0,b_1]\subset (mr_\varepsilon,1)$
for which there exists a constant $\eta>0$ satisfying
$$
b+mr_1\cos\varphi_1\ge \eta,
\qquad
\forall (r_1,\varphi_1)\in D,
\qquad
\forall b\in B.
$$

Hence the logarithmic singularity remains uniformly separated from the
trapping region and the return map is $C^k$-regular on a neighbourhood
of $D$.
 The first result establishes the uniform contraction of the radial
component $\mathcal{G}_b$ defined in \eqref{def_Rb}.
This contraction is responsible for collapsing the two-dimensional
dynamics onto an attracting invariant curve, providing the
one-dimensional reduction used in \cite[Sec.~4]{BS93}.

\begin{lemma}
\label{lem:radial_contraction} 

Let $\mathcal{G}_b:I\times\mathbb S^1\to I$ be the radial component of
\eqref{def_Rb}.
If $m\ll\lambda$, then there exists a constant
$\kappa\in(0,1)$ such that

\[
\left|
\frac{\partial \mathcal{G}_b}{\partial r_1}
(r_1,\varphi_1)
\right|
\le \kappa,
\qquad
\forall (r_1,\varphi_1)\in I\times\mathbb S^1 .
\]
\end{lemma}
In particular, for every fixed $\varphi_1\in\mathbb S^1$, the map
$r_1\mapsto \mathcal{G}_b(r_1,\varphi_1)$
is a uniform contraction.
\begin{proof}

Recall that

\[
\mathcal{G}_b(r_1,\varphi_1)
=
\left[
\frac{\alpha}{\varepsilon}
+
\left(
1-\frac{\alpha}{\varepsilon}
\right)
z_2^{2\varepsilon/\lambda}
\right]^{-1/2},
\]
where $z_2=b+mr_1\cos\varphi_1.$
Writing

\[
u(r_1,\varphi_1)
=
\frac{\alpha}{\varepsilon}
+
\left(
1-\frac{\alpha}{\varepsilon}
\right)
z_2^{2\varepsilon/\lambda},
\]
we have
$\mathcal{G}_b=u^{-1/2}.$
Differentiating with respect to $r_1$, we obtain
$$
\frac{\partial \mathcal{G}_b}{\partial r_1}
=
-\frac12
u^{-3/2}
\left(
1-\frac{\alpha}{\varepsilon}
\right)
\frac{2\varepsilon}{\lambda}
z_2^{2\varepsilon/\lambda-1}
m\cos\varphi_1.
$$
Therefore
\[
\left|
\frac{\partial G_b}{\partial r_1}
\right|
=
\frac{m}{\lambda}
|\alpha-\varepsilon|
\,|\cos\varphi_1|\,
z_2^{2\varepsilon/\lambda-1}
u^{-3/2}.
\]
Since $z_2\ge \eta>0$
and $D$ is compact, the quantities
$z_2^{2\varepsilon/\lambda-1}$ and $u^{-3/2}$ are uniformly bounded.
Hence there exists $C>0$ such that
$
\left|
\frac{\partial \mathcal{G}_b}{\partial r_1}
(r_1,\varphi_1)
\right|
\le
C\,\frac{m}{\lambda}.
$
Since $m\ll\lambda$, choosing $m/\lambda$ sufficiently small yields
$C\frac{m}{\lambda}<1. $
Setting
$\kappa=C\frac{m}{\lambda},$ the result follows. \\
\end{proof}

\begin{lemma}
\label{thm:invariant_graph}
Assume that $m\ll\lambda\ll1$ and let
$
\mathcal R_b=(\mathcal G_b,\mathcal F_b)
$
denote the return map on the trapping region
$
D=I\times\mathbb S^1.
$
Then there exists a unique invariant graph
$
\Gamma_b
=
\{(\gamma_b(\varphi_1),\varphi_1):
\varphi_1\in\mathbb S^1\},
$
with $\gamma_b\in C^3(\mathbb S^1)$, satisfying
\[
\gamma_b
\Bigl(
\mathcal F_b(\gamma_b(\varphi_1),\varphi_1)
\Bigr)
=
\mathcal G_b(\gamma_b(\varphi_1),\varphi_1),
\qquad
\forall\,\varphi_1\in\mathbb S^1.
\]
Moreover, $\Gamma_b$ is attracting: there exist constants
$C>0$ and $\rho\in(0,1)$ such that
\[
\operatorname{dist}
\bigl(
\mathcal R_b^n(r_0,\varphi_0),
\Gamma_b
\bigr)
\le
C\rho^n,
\]
for every forward orbit that remains in $D$.
Furthermore, if
\[
r^*(b)
=
\left[
\frac{\alpha}{\varepsilon}
+
\left(
1-\frac{\alpha}{\varepsilon}
\right)
b^{2\varepsilon/\lambda}
\right]^{-1/2},
\]
then
$
\|\gamma_b-r^*(b)\|_{C^3}
=
\mathcal O(m)
$
uniformly on compact parameter intervals
$B\subset(mr_\varepsilon,1)$. The map
\[
b\mapsto\gamma_b
\]
is of class $C^1$ as a map into
$C^3(\mathbb S^1)$, and there exists
$C_\gamma>0$ such that
$
\|\partial_b\gamma_b\|_{C^0}
\le
C_\gamma .
$
\end{lemma}

\begin{proof}

By Lemma~\ref{lem:radial_contraction}, there exists
$\kappa\in(0,1)$ such that
$
\sup_D
\left|
\frac{\partial\mathcal G_b}{\partial r_1}
\right|
\le\kappa .
$
Hence, for every fixed $\varphi_1\in\mathbb S^1$, the fibre map
$
r_1\longmapsto \mathcal G_b(r_1,\varphi_1)
$
is a uniform contraction on $I$.

Since $
b+m r_1\cos\varphi_1\ge\eta>0 $
throughout $D$, the logarithmic singularity remains uniformly separated
from the trapping region. Consequently,
$\mathcal R_b$ extends as a $C^3$ map to a neighbourhood of $D$.
 By Lemma 7.1, the fibres
$r\mapsto \mathcal{G}_b(r,\varphi_1)$
are uniformly contracting with contraction factor
$\kappa<1$.
Since $\mathcal{R}_b$ extends as a $C^3$ map to a neighbourhood of D
and D is positively invariant,
the fibre contraction theorem
of  \cite[Ch. 3]{HPS}
 applies to the associated graph transform.
Therefore there exists a unique invariant graph
$\gamma_b\in C^3(\EU^1).$
 Hence there exists a unique $C^3$ invariant curve
$$
\gamma_b:\mathbb S^1\to I
$$
whose graph
$
\Gamma_b=\operatorname{graph}(\gamma_b)
$
is invariant under $\mathcal R_b$. The same theorem yields exponential attraction to the invariant graph.
Therefore there exist constants
$C>0$ and $\rho\in(0,1)$ such that
\[
\operatorname{dist}
\bigl(
\mathcal R_b^n(r_0,\varphi_0),
\Gamma_b
\bigr)
\le
C\rho^n
\]
for every orbit that remains in $D$.  Define
\[
r^*(b)
=
\left[
\frac{\alpha}{\varepsilon}
+
\left(
1-\frac{\alpha}{\varepsilon}
\right)
b^{2\varepsilon/\lambda}
\right]^{-1/2}.
\]

This is the fixed point of the radial map obtained from
$\mathcal G_b$ by setting $m=0$.
A Taylor expansion with respect to $m$ gives
\[
\mathcal G_b(r^*(b),\varphi_1)
=
r^*(b)+\mathcal O(m)
\]
uniformly in $(b,\varphi_1)$.
Using the invariance equation together with the contraction estimate for
$\mathcal G_b$, a standard graph-transform argument gives
\[
\|\gamma_b-r^*(b)\|_{C^0}
=
\mathcal O(m).
\]

To obtain the $C^3$ estimate, note that:

\begin{enumerate}
\item $\mathcal R_b$ extends as a $C^3$ map to a neighbourhood of $D$;
\item the fibre maps $r_1\mapsto\mathcal G_b(r_1,\varphi_1)$ are uniformly contracting;
\item all partial derivatives of $\mathcal R_b$ up to order three are uniformly bounded on $B\times D$;
\item for $m=0$, the invariant graph reduces to the constant graph
$r_1=r^*(b)$.
\end{enumerate}

Moreover,
\[
\mathcal G_b(r_1,\varphi_1)
=
\mathcal G_b^{(0)}(r_1)
+
\mathcal O(m)
\]
in the $C^3$ topology, where $\mathcal G_b^{(0)}$
denotes the radial map corresponding to $m=0$. Therefore the associated graph transform is an
$\mathcal O(m)$ perturbation of the graph transform whose fixed point is
$r_1=r^*(b)$.
Since the graph transform is a uniform contraction,
stability of fixed points under contractive perturbations yields
$
\|\gamma_b-r^*(b)\|_{C^3}
=
\mathcal O(m).
$

Differentiating the invariance equation up to third order then gives
$\gamma_b^{(j)}=\mathcal O(m),$ for $j=1,2,3,$
relative to the constant graph $r_1=r^*(b)$. Finally, differentiating the invariance equation with respect to the
parameter $b$ yields a linear fixed-point equation for
$\partial_b\gamma_b$.
Since the associated linear operator has norm bounded by
$\kappa<1$, one obtains
$
\|\partial_b\gamma_b\|_{C^0}
\le
C_\gamma.
$
The same argument applied to difference quotients proves that
$
b\mapsto\gamma_b
$
is of class $C^1$ as a map into $C^3(\mathbb S^1)$. \\
\end{proof}

\begin{proposition}[Reduction to the invariant curve]
\label{prop:reduction}

Let $\Gamma_b=\operatorname{graph}(\gamma_b)$ be the invariant curve
given by Lemma \ref{thm:invariant_graph}.
Then the dynamics of $\mathcal R_b$ restricted to $\Gamma_b$
is completely described by the one-dimensional map
$
f_b:\mathbb S^1\to\mathbb S^1,$ where $f_b(\varphi_1)
=
\mathcal{F}_b(\gamma_b(\varphi_1),\varphi_1).
$
Moreover,
$$
\mathcal R_b
(\gamma_b(\varphi_1),\varphi_1)
=
(\gamma_b(f_b(\varphi_1)),
f_b(\varphi_1)).
$$
\end{proposition}

\begin{proof}

Since $\Gamma_b$ is invariant under $\mathcal R_b$,
$
\mathcal R_b
(\gamma_b(\varphi_1),\varphi_1)
\in\Gamma_b.
$
Therefore there exists a unique angle
$f_b(\varphi)\in\mathbb S^1$
such that
$
\mathcal R_b
(\gamma_b(\varphi_1),\varphi_1)
=
(\gamma_b(f_b(\varphi_1)),
f_b(\varphi_1)).
$
By definition of the second component of $\mathcal R_b$,
$
f_b(\varphi_1)
=
\mathcal{F}_b(\gamma_b(\varphi_1),\varphi_1).
$
This proves the claim.

\end{proof}

\begin{remark}
\label{rem:gamma}

The compact-parameter assumption in this section is used only to
construct and control the invariant graph away from the logarithmic
singularity. The singular regime
$b\to(mr_\varepsilon)^+$
is analysed in the subsequent sections through parameter-dependent
estimates for the reduced map \(f_b\).
For every fixed admissible parameter value, however,
the asymptotic dynamics inside the trapping region is completely
encoded by the one-dimensional map \(f_b\).
\end{remark}

\begin{figure}[ht]
\begin{center}
 \includegraphics[height=6.5cm]{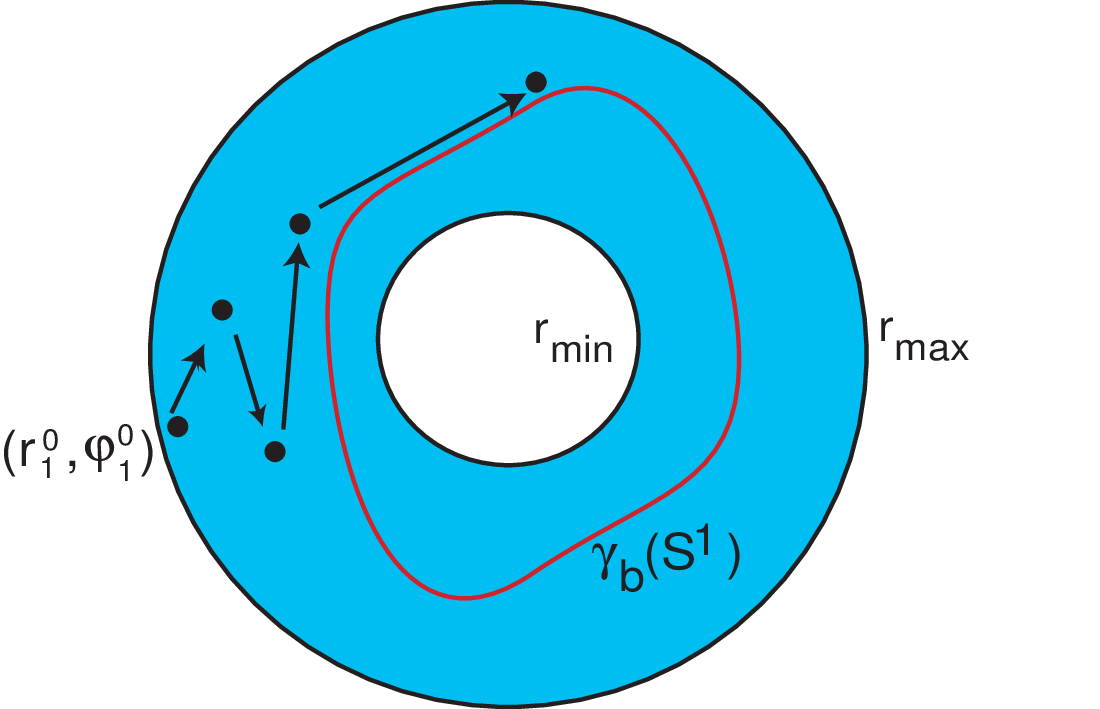}
\end{center}
\caption{\small  Geometric representation of the invariant curve $\gamma_b(\EU^1)$ within the annulus. The sequence of iterates $(r_1^n, \varphi_1^n)$ shows exponential radial attraction toward the curve, governed by the contraction rate $\kappa \in (0,1)$. }
 \label{Imagem 5}
\end{figure}

\section{An overview of the theory of rank-one strange attractors}
\label{s: theory}

The theory of rank-one maps, systematically developed by Wang and Young \cite{WY2001, WY, WY2006, WY2003, WY2008}, concerns the dynamics of maps with some instability in one direction of the phase space and strong contraction in all other directions of the phase space. This theory   originated with the work of Jacobson \cite{Ja81} on the quadratic family and the  analysis of strongly dissipative H\'enon maps by Benedicks and Carleson \cite{BC91}.  The method can be seen as a generalisation of \cite[Ch. 5]{DRV}.

It is a comprehensive theory for a nonuniformly hyperbolic setting that is flexible enough to be applicable to concrete systems of differential equations and has experienced unprecedented
growth in the last 20 years in the context of non-autonomous systems. It provides checkable conditions that imply the existence of nonuniformly hyperbolic dynamics and  SRB measures in parametrised families $F_{\lambda}$ of dissipative embeddings in $\RR^n$ for $n \geq 2$.  
 This theory has already been applied to several non-autonomous dynamical scenarios, including systems with stable foci and  limit cycles subject to pulsatile drives  \cite{WY, WY2003} and heteroclinic bifurcations \cite{Rodrigues2021}.

This section introduces the basic objects and notions used throughout the paper. We adapt the standing hypotheses of   \cite{WY2008} to the present setting, where the effective dynamics becomes one-dimensional after strong contraction in the transverse direction.  The Euclidean metric on $\mathbb{R}$ is denoted by $d(x, y) = |x - y|$, and the distance from a point to a set $C\subset \EU^1$ is $\dpt d(x, C) = \min_{\hat{x}\in C} |x - \hat{x}|$.\\

Let $I$ be a compact interval where the terminal points are identified or the circle $\EU^1$, and let $f: I \to I$ be a $C^2$ map. Let $C := \{x \in I : f'(x) = 0\}$ be the \emph{critical set} (it contains all critical points of $f$, which we assume to be finite); the image under $f$ of a critical point is a \emph{critical value}. \\ 

For $\delta > 0$, denote by $C_\delta := \{x \in I : d(x, C) < \delta\}$ the $\delta$-neighbourhood of $C$ -- sometimes called \emph{``dangerous''} or \emph{``critical''} zone.
 We now introduce the meaning of Misiurewicz map.
\begin{definition}[\cite{WY2008}, adapted]
\label{def-Misiur}
We say that $f$ is a \emph{Misiurewicz map} if there exists $\delta_0 > 0$ such that:
\bigbreak
\noindent
 (a) Non-recurrence of critical orbits:
For every critical point $\hat{x} \in C$, $d(f^n(\hat{x}), C) > 2\delta_0$ for all $n \ge 1$.\\

\smallbreak
\noindent
(b) Expansion outside the critical neighbourhood:
There exist constants $\lambda_0 > 0$, $M_0 \in \mathbb{N}$, and $0 < c_0 \le 1$ such that:
\begin{itemize}
    \item[(b1)] For all $n \ge M_0$, if $x, f(x), \ldots, f^{n-1}(x) \notin C_{\delta_0}$, then $|(f^n)'(x)| \ge e^{\lambda_0 n}$.
    \item[(b2)] For all $n \ge 1$, if $x, f(x), \ldots, f^{n-1}(x) \notin C_{\delta_0}$ and $f^n(x) \in C_{\delta_0}$, then $|(f^n)'(x)| \ge c_0 e^{\lambda_0 n}$.\\
\end{itemize}
\smallbreak
\noindent
(c) Behaviour inside the critical neighbourhood:
There exists $K_0 > 1$ such that for all $x \in C_{\delta_0}$:
\begin{itemize}
    \item[(c1)] $f''(x) \neq 0$ (the critical points are non-degenerate).
    \item[(c2)] There exists an integer $p(x) \ge 1$ satisfying $K_0^{-1} \log \frac{1}{d(x,C)} \le p(x) \le K_0 \log \frac{1}{d(x,C)}$ such that $f^j(x) \notin C_{\delta_0}$ for $0 \le j < p(x)$, and $|(f^{p(x)})'(x)| \ge c_0^{-1} e^{\frac{1}{3}\lambda_0 p(x)}$.\\
\end{itemize}
\end{definition}

\subsection*{Digestive remark about Definition \ref{def-Misiur}}
 Property \textbf{(a)} 
ensures that once a critical point  is achieved, its 
forward orbit never returns to a neighbourhood of any critical point. This 
prevents the system from accumulating zero-derivative effects, which would 
otherwise destroy the expansion needed for chaos. 

Properties \textbf{(b1)} and \textbf{(b2)} guarantee that the map is a powerful 
``stretcher" outside the critical zones $C_\delta$, $\delta > 0$. Any small interval of points is 
stretched exponentially as long as it stays in the safe region. While 
\textbf{(b1)} covers orbits that end in the safe region, \textbf{(b2)} allows 
for a slight loss of momentum (represented by the constant $c_0>0$) when an 
orbit lands right at the doorstep of a critical neighbourhood, provided it 
maintained exponential growth until that point.

Finally, property \textbf{(c)} describes the ``recovery phase" for orbits that 
fall into the critical zone. Condition \textbf{(c1)} ensures the fold 
is non-degenerate, providing a predictable curvature. Condition 
\textbf{(c2)} defines the  \emph{binding mechanism}: if a point lands very 
near a critical point, it must ``shadow" the critical value's orbit for a 
specific binding time $p(x)$. During this period, the point is protected 
from falling into other critical zones, and by the end of it, the initial 
loss of derivative is fully compensated, allowing the orbit to emerge with 
a renewed exponential boost.\\

Now we state the main result that will be used to prove Theorem \ref{main1}.  The set $I\subset \RR^+$ is a compact interval of $\RR$.\\

\begin{theorem}[Rank-one reduction criterion, \cite{WY2008}, adapted]
\label{thm:WY}
Let $\mathcal{R}_b: I \times \EU^1 \to I \times \EU^1$ be a family of $C^3$ embeddings defined by $$\mathcal{R}_b(r_1, \varphi_1) = (\mathcal{G}_b(r_1, \varphi_1), \mathcal{F}_b(r_1, \varphi_1)),$$ where $b \in [b_0, b_1]\subset \RR^+$ and $b_0<b_1$. Suppose:\\
\begin{itemize}
    \item[\textbf{(H1)}]  {Strong transverse contraction:} $|\frac{\partial \mathcal{G}_b}{\partial r_1}(r_1, \varphi_1)| \le \kappa < 1$ for all $(r_1, \varphi_1)\in I \times \EU^1$.\\
   \item[\textbf{(H2)}] {Invariant graph:}
There exists a unique $C^3$ invariant graph
$
r_1=\gamma_b(\varphi_1),
$
whose graph
\[
\Gamma_b=
\{(\gamma_b(\varphi_1),\varphi_1):
\varphi_1\in\mathbb S^1\}
\]
is invariant under $\mathcal R_b$. Equivalently,
\[
\gamma_b\!\left(
\mathcal{F}_b(\gamma_b(\varphi_1),\varphi_1)
\right)
=
\mathcal{G}_b(\gamma_b(\varphi_1),\varphi_1),
\qquad
\forall\,\varphi_1\in\mathbb S^1.
\]
    \item[\textbf{(H3)}] {Misiurewicz dynamics:} The reduced map $f_b(\varphi_1) = \mathcal{F}_b(\gamma_b(\varphi_1), \varphi_1)$ is a $C^3$-Misiurewicz map for some $b^*\in [b_0, b_1]$ (according to Definition \ref{def-Misiur}).\\
    \item[\textbf{(H4)}] {Parameter transversality:} If $\hat{\varphi}_1(b)$ is a critical point of $f_b$, then the speed of the critical value $\xi(b) = f_b(\hat{\varphi}_1(b))$(\footnote{Continuation of the critical value  as function of $b$.}) satisfies $\frac{d}{db}|_{b=b^*} \xi(b) \neq 0$. \\
    \item[\textbf{(H5)}] {Non-degeneracy:} For every critical point $\hat{\varphi}_1$ of $f_{b^*}$, $\frac{\partial \mathcal{F}_{b^*}}{\partial r_1}(\gamma_{b^*}(\hat{\varphi}_1), \hat{\varphi}_1) \neq 0$.\\
\end{itemize}
Then there exists a set $\Delta \subset (b_0, b_1)$ of positive Lebesgue measure such that for $b \in \Delta$, $\mathcal{R}_b$ admits a ``large'' strange attractor $\Lambda_b$ with an ergodic SRB measure.\\
\end{theorem}

Theorem \ref{thm:WY} states that if the map  $\mathcal{R}_b$
flattens the space strongly in one direction (\textbf{(H1)--(H2)}), 
possesses a well-behaved ``stretch-and-fold" dynamics in the other (\textbf{H3}), 
and moves its critical points  as the parameter $b$ varies (\textbf{H4}), 
then the system inevitably collapses onto a strange attractor, say $\Lambda_b$.   The dynamics are genuinely non-uniformly hyperbolic: a Central Limit Theorem holds and correlations decay at an exponential rate (cf. \cite{WY}). 

\begin{remark} 
Theorem \ref{thm:WY} is a reformulation of the
rank-one criterion of Wang--Young  \cite{WY2008} adapted to the
present setting, where the normally attracting
manifold is represented by the invariant graph
$r=\gamma_b(\varphi_1)$ instead of the reference
curve $y=0$.
\end{remark}

 \section{Preparatory results}
\label{s: Preparatory}
In this section, we establish several preparatory estimates concerning the regularity and parameter dependence of the radial component. In the following, $\mathcal{G}_b$ denotes the radial component of the first return map $\mathcal{R}_b$ as derived in Proposition \ref{thm:0}, and $\gamma_b(\varphi_1)$ is the invariant curve whose existence and attractivity were proved in Lemma \ref{thm:invariant_graph}.The following result establishes that the derivatives of the invariant curve and the radial map remain controlled by the small perturbation parameter $m$. The following estimates are understood on compact parameter intervals
\(B\subset(mr_\varepsilon,1)\), where the logarithmic singularity is
uniformly avoided. In what follows, we are going to assume that  $B=[b_0,b_1]\subset (m r_\varepsilon,1)$ is a compact parameter interval such that
$$
b+m r_1\cos\varphi_1\geq \eta>0,
\qquad
\forall (r_1,\varphi_1)\in D,
\quad
\forall b\in B.
$$

\bigbreak

 \begin{lemma}
\label{Lemma Aux1}
Assume that $m\ll \lambda.$ Then, for \(b\in B\) and \(\varphi_1\in\mathbb S^1\), the
following estimates hold:\\

\begin{enumerate}

\item
$
\partial_{\varphi_1}
\mathcal{G}_b(\gamma_b(\varphi_1),\varphi_1)
=
\mathcal O\!\left(\frac{m}{\lambda}\right).
$ \\
\item
$
\gamma_b'(\varphi_1)
=
\mathcal O(m).
$
\\
\item
$
\gamma_b''(\varphi_1)
=
\mathcal O(m).
$
\\
\end{enumerate}
The estimates in \emph{(2)} and \emph{(3)} are uniform for
\(b\in B\).
\end{lemma}

\begin{proof}

\begin{enumerate}

\item
Recall that
$\mathcal{G}_b(r_1,\varphi_1)
=
u(r_1,\varphi_1)^{-1/2},$
where

\begin{eqnarray*}
u(r_1,\varphi_1) &=&\frac{\alpha}{\varepsilon} +\left(1-\frac{\alpha}{\varepsilon}\right)z_2^{2\varepsilon/\lambda} \\
z_2&=&b+m r_1\cos\varphi_1.
\end{eqnarray*}
Since \(0<\varepsilon<\alpha\), this may equivalently be written as
\[
u(r_1,\varphi_1)
=
\frac{\alpha}{\varepsilon}
-
\left(
\frac{\alpha}{\varepsilon}-1
\right)
z_2^{2\varepsilon/\lambda}.
\]
One knows that $z_2=b+m r_1\cos\varphi_1\geq \eta>0$
on \(B\times D\). Moreover \(D\) is compact and, after restricting the
trapping region if necessary, \(z_2\leq M<1\). Hence \(u\) is bounded
away from zero. Indeed,
\[
u(r_1,\varphi_1)
\geq
\frac{\alpha}{\varepsilon}
-
\left(
\frac{\alpha}{\varepsilon}-1
\right)
M^{2\varepsilon/\lambda}
=:c_0>0.
\]
Therefore \(u^{-1}\) is uniformly bounded on \(B\times D\).
Differentiating \(\mathcal{G}_b\) with respect to \(\varphi_1\), we obtain
\[
\partial_{\varphi_1}\mathcal{G}_b(r_1, \varphi_1)
=
-\frac12
u^{-3/2}
\left(
1-\frac{\alpha}{\varepsilon}
\right)
\frac{2\varepsilon}{\lambda}
z_2^{2\varepsilon/\lambda-1}
(-m r_1\sin\varphi_1).
\]
Consequently,
\[
\left|
\partial_{\varphi_1}\mathcal{G}_b(r_1,\varphi_1)
\right|
\leq
C
\frac{m}{\lambda}
\,
z_2^{2\varepsilon/\lambda-1},
\]
for some constant \(C>0\), independent of \(m\), \(b\), \(r_1\) and
\(\varphi_1\).
Since \(0<z_2\leq M<1\) and \(z_2\geq \eta\), we have
\[
z_2^{2\varepsilon/\lambda-1}
=
z_2^{-1}z_2^{2\varepsilon/\lambda}
\leq
\eta^{-1},
\]
because \(z_2^{2\varepsilon/\lambda}\leq 1\). Hence
$
\left|
\partial_{\varphi_1}\mathcal{G}_b(r_1,\varphi_1)
\right|
\leq
C_1\frac{m}{\lambda},
$
uniformly on \(B\times D\). Evaluating at
\(r_1=\gamma_b(\varphi_1)\) gives
\[
\partial_{\varphi_1}
\mathcal{G}_b(\gamma_b(\varphi_1),\varphi_1)
=
\mathcal O\!\left(\frac{m}{\lambda}\right).
\]
\\

\item
By Lemma \ref{thm:invariant_graph}, one has
$
\|\gamma_b-r^*(b)\|_{C^3}
=
\mathcal O(m),
$
uniformly for \(b\in B\), where
$$
r^*(b)
=
\left[
\frac{\alpha}{\varepsilon}
+
\left(
1-\frac{\alpha}{\varepsilon}
\right)
b^{2\varepsilon/\lambda}
\right]^{-1/2}.
$$
In particular,
$\|\gamma_b-r^*(b)\|_{C^2}
=
\mathcal O(m).
$
Since \(r^*(b)\) is independent of \(\varphi_1\), we have
$(r^*(b))'=0.$ Therefore
$
\gamma_b'
=
(\gamma_b-r^*(b))',
$
and so
$
\|\gamma_b'\|_\infty
\leq
\|\gamma_b-r^*(b)\|_{C^2}
=
\mathcal O(m).
$
Thus
$$
\gamma_b'(\varphi_1)
=
\mathcal O(m),
\qquad
\forall \varphi_1\in\mathbb S^1.
$$

\item
Again, since $r^*(b)$ is independent of \(\varphi_1\),
$(r^*(b))''=0.$
Hence
$
\gamma_b''
=
(\gamma_b-r^*(b))''.
$
Therefore
$
\|\gamma_b''\|_\infty
\leq
\|\gamma_b-r^*(b)\|_{C^2}
=
\mathcal O(m).
$
Consequently,
$
\gamma_b''(\varphi_1)
=
\mathcal O(m),
\forall \varphi_1\in\mathbb S^1.
$

\end{enumerate}

\end{proof}

Reminding Proposition \ref{prop:reduction}, the reduced one-dimensional map $f_b: \mathbb{S}^1 \to \mathbb{S}^1$, obtained by restricting the map $\mathcal{R}_b$ to the unique invariant graph $r_1 = \gamma_b(\varphi_1)$, is given by:
\begin{equation}
\label{f_b def} 
f_b(\varphi_1) = \mathcal{F}_b(\gamma_b(\varphi_1), \varphi_1) = \gamma_b(\varphi_1) A(\varphi_1) - \frac{1}{\lambda} \ln E(\varphi_1)  \pmod{2\pi},
\end{equation}
where:
\begin{eqnarray*}
A(\varphi_1) &=& a_{11} \cos \varphi_1 + a_{12} \sin \varphi_1\\
E(\varphi_1) &=& b + m \gamma_b(\varphi_1) \cos \varphi_1.\\
\end{eqnarray*}
  The map $f_b$ captures the dynamics on the attracting invariant curve given in Lemma \ref{thm:invariant_graph}.\\

\begin{lemma}
\label{Lemma Aux2}
Let
$
f_b(\varphi_1)
=
\mathcal{F}_b(\gamma_b(\varphi_1),\varphi_1)
$
be the reduced one-dimensional map.  Assume that
$
m\ll\lambda\ll1.
$
Then
$$
f_b'(\varphi_1)
=
\frac{m}{\lambda}
\frac{\gamma_b(\varphi_1)\sin\varphi_1}
     {E_b(\varphi_1)}
+
R_b(\varphi_1),
$$
where
$
R_b(\varphi_1)
=
\gamma_b'(\varphi_1)A(\varphi_1)
+
\gamma_b(\varphi_1)A'(\varphi_1)
-
\frac{m}{\lambda}
\frac{\gamma_b'(\varphi_1)\cos\varphi_1}
     {E(\varphi_1)}.
$
Furthermore (\footnote{Here $\|R_b\|_\infty$ denotes the usual sup-norm $
\|R_b\|_\infty
= \dpt
\sup_{\varphi_1\in \EU^1}|R_b(\varphi_1)|.
$}):
\[
\|R_b\|_\infty
\leq
C
\left(
1+
\dpt \frac{m^2}{\lambda\,\inf_{\varphi_1\in\mathbb S^1}E_b(\varphi_1)}
\right),
\]
where \(C>0\) is independent of \(\varphi_1\).
In particular, since \(b\) belongs to a compact   interval
$B\subset(mr_\varepsilon,1) $
for which
\[
E(\varphi_1)\geq \eta>0,
\qquad
\forall \varphi_1\in\mathbb \EU^1,
\quad
\forall b\in B,
\]
then
$
\|R_b\|_\infty=\mathcal O(1),
$
uniformly for \(b\in B\).  
\end{lemma}

For each fixed \(b>mr_\varepsilon\), the
same estimate holds uniformly in \(\varphi_1\), although the implied
constant may depend on the distance of \(b\) to the singular threshold
\((mr_\varepsilon)^+\).

\begin{proof}

By definition, one has $f_b(\varphi_1)=
\gamma_b(\varphi_1)A(\varphi_1)
-
\frac{1}{\lambda}\log E(\varphi_1),$ 
where
\[
E(\varphi_1)
=
b+m\gamma_b(\varphi_1)\cos\varphi_1 .
\]
Differentiating with respect to \(\varphi_1\), we obtain
\[
f_b'(\varphi_1)
=
\gamma_b'(\varphi_1)A(\varphi_1)
+
\gamma_b(\varphi_1)A'(\varphi_1)
-
\frac1\lambda
\frac{E'(\varphi_1)}{E(\varphi_1)}.
\]
Since $E'(\varphi_1) = m\gamma_b'(\varphi_1)\cos\varphi_1 - m\gamma_b(\varphi_1)\sin\varphi_1,$ it follows that
\[
\begin{aligned}
f_b'(\varphi_1)
&=
\gamma_b'(\varphi_1)A(\varphi_1)
+
\gamma_b(\varphi_1)A'(\varphi_1)
\\
&\quad
-
\frac1\lambda
\frac{
m\gamma_b'(\varphi_1)\cos\varphi_1
-
m\gamma_b(\varphi_1)\sin\varphi_1
}
{E(\varphi_1)}
\\
&=
\frac{m}{\lambda}
\frac{\gamma_b(\varphi_1)\sin\varphi_1}
     {E(\varphi_1)}
+
R_b(\varphi_1),
\end{aligned}
\]
where
\[
R_b(\varphi_1)
=
\gamma_b'(\varphi_1)A(\varphi_1)
+
\gamma_b(\varphi_1)A'(\varphi_1)
-
\frac{m}{\lambda}
\frac{\gamma_b'(\varphi_1)\cos\varphi_1}
     {E(\varphi_1)}.
\]

We now estimate the remainder. By Lemmas~\ref{thm:invariant_graph}
and~\ref{Lemma Aux1}, one has:
$
\gamma_b(\varphi_1)=\mathcal O(1) $ and $\gamma_b'(\varphi_1)=\mathcal O(m),$
uniformly in \(\varphi_1\). Since \(A\) and \(A'\) are smooth functions on \(\mathbb S^1\), they are uniformly bounded. Therefore,
\[
\gamma_b'(\varphi_1)A(\varphi_1)
+
\gamma_b(\varphi_1)A'(\varphi_1)
=
\mathcal O(1).
\]

For the last term in \(R_b\),
\[
\left|
\frac{m}{\lambda}
\frac{\gamma_b'(\varphi_1)\cos\varphi_1}
     {E(\varphi_1)}
\right|
\le
C
\frac{m}{\lambda}
\frac{m}{E(\varphi_1)}
=
C
\frac{m^2}
     {\lambda\,E(\varphi_1)}.
\]

Hence,
\[
\|R_b\|_\infty
\le
C
\left(
1+
\frac{m^2}
     {\lambda\,\inf_{\varphi_1\in\mathbb S^1}E(\varphi_1)}
\right).
\]
Since \(b\in B\subset (mr_\varepsilon,1)\) and
\[
E(\varphi_1)\ge \eta>0,
\qquad
\text{for all}\, \,  (\varphi_1,b)\in\mathbb S^1\times B,
\]
then
\[
\frac{m^2}
     {\lambda\,\inf_{\varphi_1}E(\varphi_1)}
\le
\frac{m^2}{\lambda\eta}.
\]

Since \(m\ll\lambda\), one has \(m^2/\lambda=o(1)\). Thus
$\|R_b\|_\infty=\mathcal O(1)$ uniformly for \(b\in B\).
For each fixed parameter value \(b>mr_\varepsilon\), the quantity
$
\inf_{\varphi_1\in\mathbb S^1}E(\varphi_1)
$
is strictly positive, and therefore the same estimate holds uniformly in \(\varphi_1\).
This completes the proof.

\end{proof}

\begin{remark}
\label{rem:remainder_not_uniform_singular}

The estimate \(\|R_b\|_\infty=\mathcal O(1)\) is uniform only on
compact parameter intervals separated from the singular threshold.
If \(b\to(mr_\varepsilon)^+\), then
$\dpt \inf_{\varphi\in\mathbb S^1}E_b(\varphi)
$ may tend to zero, and the above bound must be used in its explicit
form
\[
\|R_b\|_\infty
\leq
C
\left(
1+
\frac{m^2}{\lambda\,\inf E}
\right).
\]
 
\end{remark}

Lemma \ref{Lemma Aux2} shows that the only contribution
exhibiting the singular scaling
$
\lambda^{-1}E(\varphi_1)^{-1}
$
arises from the logarithmic term. In particular,
although $m/\lambda$ is small, the quantity
\[
\frac{m}{\lambda}
\frac{\gamma_b(\varphi_1)\sin\varphi_1}
     {b+m\gamma_b(\varphi_1)\cos\varphi_1}
\]
can become arbitrarily large as $b\to (mr_\varepsilon)^+$.
In this regime,
$
E(\varphi_1)
=
b+m\gamma_b(\varphi_1)\cos\varphi_1
$
may become arbitrarily small near
\(\varphi_1\approx\pi\).
Consequently,
$
\frac{1}{\lambda}
\frac{E'(\varphi_1)}
     {E(\varphi_1)}
$
becomes arbitrarily large in absolute value.
This produces the sharp peaks generated by the
logarithmic singularity  (near $0$ and $\pi$ as depicted in Figure \ref{Imagem 4}) and identifies the
leading mechanism responsible for the strong
expansion analysed in the subsequent sections.
 
 In particular, outside a sufficiently small
neighbourhood of the critical points, the quantity
$|\sin\varphi_1|$ is bounded away from zero, while
the denominator $E(\varphi_1)$ remains strictly
positive. Hence the dominant term
$
\frac{m}{\lambda}
\frac{\gamma_b(\varphi_1)\sin\varphi_1}
     {E(\varphi_1)}
$
retains a definite sign and magnitude, providing
the quantitative source of the uniform expansion
established later in Lemma~\ref{lem:uniform_expansion}.
 
 \begin{remark}
 Throughout the singular regime under consideration,
\(b\gtrsim mr_\varepsilon\) implies that \(b\) remains
comparable with \(m\); equivalently, there exist
constants \(C_1,C_2>0\) independent of \(m\) such that
\[
C_1m\le b\le C_2m.
\]
 \end{remark}

\begin{figure}[ht]
\begin{center}
 \includegraphics[height=8.5cm]{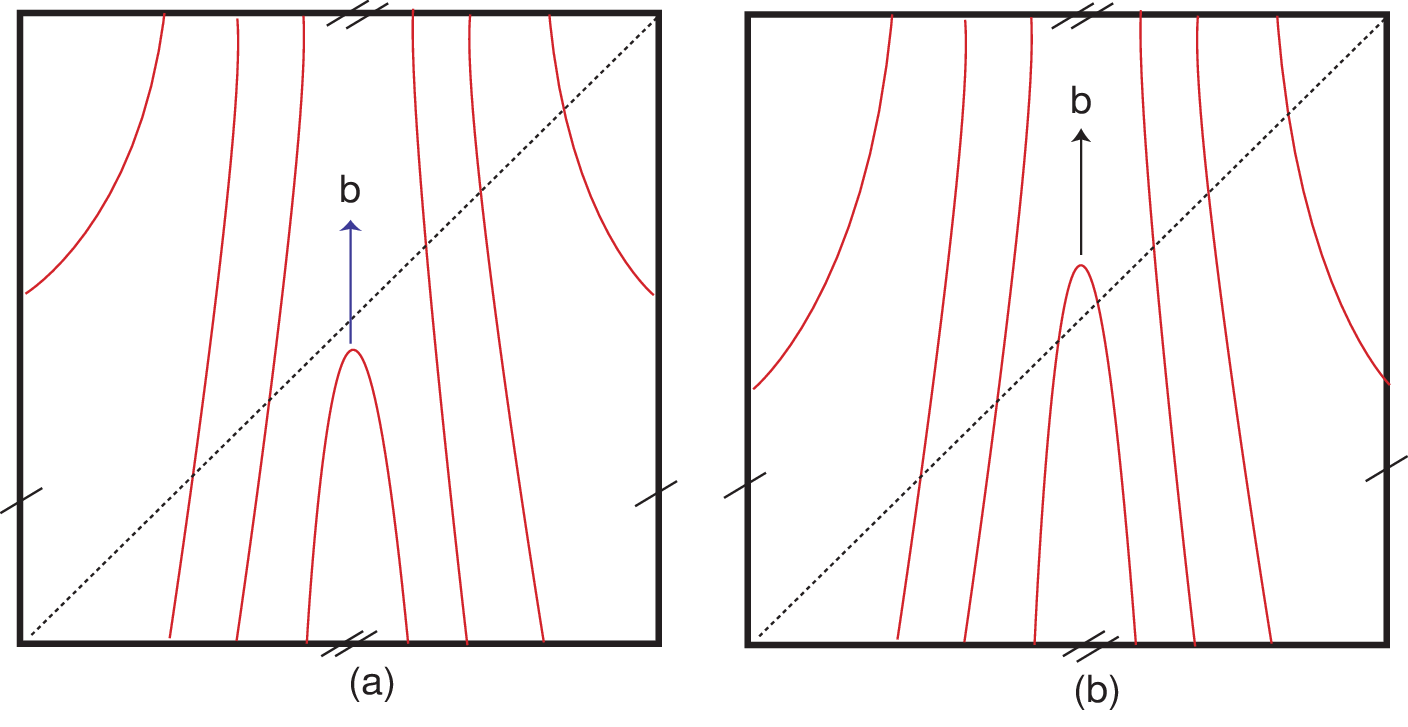}
\end{center}
\caption{\small  Sketch of the graph of $f_b: \EU^1\to \EU ^1$  for different values of $b>mr_\varepsilon$.  The bars identify the sides of the square. The critical points vary according to $b$. }
 \label{Imagem 4}
\end{figure}

\begin{lemma}[Non-degenerate critical points]
\label{lem:nondegeneracy}

Suppose that
$m\ll\lambda\ll1$ and $ b \gtrsim m r_\varepsilon$
Let \(\widehat\varphi_1\) be a critical point of the reduced map
\(f_b\). Then \(\widehat\varphi_1\) is non-degenerate.  
\end{lemma}

\begin{proof}

Recall that
$
f_b(\varphi_1)
=
\gamma_b(\varphi_1)A(\varphi_1)
-
\frac1\lambda
\log E(\varphi_1),
$
where
$
A(\varphi_1)=a_{11}\cos\varphi_1+a_{12}\sin\varphi_1
$
and
$
E(\varphi_1)
=
b+m\gamma_b(\varphi_1)\cos\varphi_1.
$
Throughout the proof, all estimates are evaluated in the domain of the
return map, so that
$
E(\varphi_1)>0.
$
Differentiating $f_b$ once  gives
\[
f_b'(\varphi_1)
=
(\gamma_b A)'(\varphi_1)
-
\frac1\lambda
\frac{E'(\varphi_1)}{E(\varphi_1)}. 
\]
Hence, at a critical point \(\widehat\varphi_1\), we have:
\[
0=f_b'(\widehat\varphi_1)
=
(\gamma_b A)'(\widehat\varphi_1)
-
\frac1\lambda
\frac{E'(\widehat\varphi_1)}
     {E(\widehat\varphi_1)}.
\]
Therefore
\[
E'(\widehat\varphi_1)
=
\lambda E(\widehat\varphi_1)
(\gamma_b A)'(\widehat\varphi_1). 
\]
Since \(A\) and \(A'\) are uniformly bounded and, by Lemmas~\ref{thm:invariant_graph}
and~\ref{Lemma Aux1}, one has: $
\gamma_b=\mathcal O(1)$ and $\gamma_b'=\mathcal O(m)$, we have
$(\gamma_b A)'=\mathcal O(1).$
Consequently,
\[
E'(\widehat\varphi_1)
=
\mathcal O(\lambda E(\widehat\varphi_1)).
\tag{9.3}
\label{eq:Eprime-critical}
\]
On the other hand,
\[
E'(\varphi_1)
=
m\gamma_b'(\varphi_1)\cos\varphi_1
-
m\gamma_b(\varphi_1)\sin\varphi_1.
\]
Evaluating at \(\widehat\varphi_1\) and using
\eqref{eq:Eprime-critical}, we obtain
\[
m\gamma_b(\widehat\varphi_1)\sin\widehat\varphi_1
=
m\gamma_b'(\widehat\varphi_1)\cos\widehat\varphi_1
+
\mathcal O(\lambda E(\widehat\varphi_1)).
\]
Since \(\gamma_b\) is bounded away from zero on the trapping annulus,
\[
\gamma_b(\widehat\varphi_1)\ge r_{\min}>0,
\]
and since \(\gamma_b'=\mathcal O(m)\), it follows that
\[
\sin\widehat\varphi_1
=
\mathcal O(m)
+
\mathcal O\!\left(
\frac{\lambda E(\widehat\varphi_1)}{m}
\right).
\]
Using $b\gtrsim mr_\varepsilon$ and the boundedness of \(\gamma_b\), we have
\[
0<E(\widehat\varphi_1)
=
b+m\gamma_b(\widehat\varphi_1)\cos\widehat\varphi_1
\le C m
\]
for some constant \(C>0\). Hence
\[
\sin\widehat\varphi_1
=
\mathcal O(m)+\mathcal O(\lambda).
\]
Since \(m\ll\lambda\), this gives $
\sin\widehat\varphi_1=\mathcal O(\lambda).$
Thus every critical point lies in an \(\mathcal{O}(\lambda)\)-neighbourhood of
either \(0\) or \(\pi\). In particular,
$
|\cos\widehat\varphi_1|
\ge
\frac12
$
for all sufficiently small \(\lambda\).
We now estimate the second derivative. Differentiating twice,
\[
f_b''(\varphi_1)
=
(\gamma_b A)''(\varphi_1)
-
\frac1\lambda
\left(
\frac{E''(\varphi_1)}{E(\varphi_1)}
-
\frac{(E'(\varphi_1))^2}{E(\varphi_1)^2}
\right).
\]
Equivalently,
\[
f_b''(\varphi_1)
=
(\gamma_b A)''(\varphi_1)
-
\frac{E''(\varphi_1)}{\lambda E(\varphi_1)}
+
\frac{1}{\lambda}
\left(
\frac{E'(\varphi_1)}{E(\varphi_1)}
\right)^2.
\tag{9.4}
\label{eq:fb-second}
\]
By Lemma~\ref{Lemma Aux1}, one has
$\gamma_b'=\mathcal O(m),$ and $\gamma_b''=\mathcal O(m),$
and therefore
\[
(\gamma_b A)''=\mathcal O(1).
\tag{9.5}
\label{eq:gammaA-bounded}
\]
Moreover,
\[
E''(\varphi_1)
=
m\gamma_b''(\varphi_1)\cos\varphi_1
-
2m\gamma_b'(\varphi_1)\sin\varphi_1
-
m\gamma_b(\varphi_1)\cos\varphi_1.
\]
Evaluating at \(\widehat\varphi_1\), and using
\[
\gamma_b'=\mathcal O(m),
\qquad
\gamma_b''=\mathcal O(m),
\qquad
\sin\widehat\varphi_1=\mathcal O(\lambda),
\]
we get
\[
E''(\widehat\varphi_1)
=
-
m\gamma_b(\widehat\varphi_1)\cos\widehat\varphi_1
+
\mathcal O(m^2).
\tag{9.6}
\label{eq:Esecond-critical}
\]
Since \(\gamma_b(\widehat\varphi_1)\ge r_{\min}>0\) and
$|\cos\widehat\varphi_1|\ge \frac12,$
we may choose \(m>0\) sufficiently small so that
\[
|E''(\widehat\varphi_1)|
\ge
c_0 m
\]
for some \(c_0>0\) independent of \(\lambda\).
Next, by \eqref{eq:Eprime-critical}, we may write $
\frac{E'(\widehat\varphi_1)}
     {E(\widehat\varphi_1)}
=
\mathcal O(\lambda).
$
Therefore
\[
\frac1\lambda
\left(
\frac{E'(\widehat\varphi_1)}
     {E(\widehat\varphi_1)}
\right)^2
=
\mathcal O(\lambda).
\tag{9.7}
\label{eq:Eprime-square}
\]
Substituting \eqref{eq:gammaA-bounded},
\eqref{eq:Esecond-critical}, and \eqref{eq:Eprime-square} into
\eqref{eq:fb-second}, we obtain
\[
f_b''(\widehat\varphi_1)
=
\frac{
m\gamma_b(\widehat\varphi_1)\cos\widehat\varphi_1
}
{
\lambda E(\widehat\varphi_1)
}
+
\mathcal O\!\left(
\frac{m^2}{\lambda E(\widehat\varphi_1)}
\right)
+
\mathcal O(1).
\]
Equivalently,
\[
f_b''(\widehat\varphi_1)
=
\frac{
m
\left[
\gamma_b(\widehat\varphi_1)\cos\widehat\varphi_1
+\mathcal O(m)
\right]
}
{
\lambda E(\widehat\varphi_1)
}
+
\mathcal O(1).
\]
Since
$
|\gamma_b(\widehat\varphi_1)\cos\widehat\varphi_1|
\ge
\frac{r_{\min}}{2},
$
the term in square brackets is bounded away from zero for \(m\) small.
Thus there exists \(c_1>0\) such that
\[
\left|
\frac{
m
\left[
\gamma_b(\widehat\varphi_1)\cos\widehat\varphi_1
+\mathcal O(m)
\right]
}
{
\lambda E(\widehat\varphi_1)
}
\right|
\ge
\frac{c_1m}{\lambda E(\widehat\varphi_1)}.
\]
Using again
$
0<E(\widehat\varphi_1)\le C m,
$
we obtain
\[
\frac{m}{\lambda E(\widehat\varphi_1)}
\ge
\frac{1}{C\lambda}.
\]
Therefore
\[
|f_b''(\widehat\varphi_1)|
\ge
\frac{c_2}{\lambda}
-
\mathcal O(1).
\]
Since \(\lambda\ll1\), the singular term dominates the bounded
remainder. Hence, for all sufficiently small \(\lambda\),
$
|f_b''(\widehat\varphi_1)|
\ge
\frac{c}{\lambda}
$
for some constant \(c>0\), independent of \(\lambda\).
In particular,
$
f_b''(\widehat\varphi_1)\neq0.
$
Thus every critical point of \(f_b\) is non-degenerate.
\end{proof}

\begin{remark}
\label{rem:nondegeneracy_does_not_count_critical_points}

Lemma~\ref{lem:nondegeneracy} is conditional on the existence of a
critical point. It proves that any such critical point is non-degenerate
and must lie near \(0\) or \(\pi\). It does not, by itself, prove that
\(f_b\) has exactly two critical points.  
\end{remark}

\begin{remark}
\label{rem:uniform_expansion_scale_separation}

The assumption
$
E(\varphi_1)\ge c_E m,
\forall \varphi_1\notin\mathcal C_{\delta_0},
$
is needed to keep the remainder \(R_b\) of Lemma~\ref{Lemma Aux2}
uniformly bounded outside the critical neighbourhood. Near the
singular threshold \(b\to(mr_\varepsilon)^+\), such a lower bound may
fail only close to the singular angular region, which is contained in
the critical neighbourhood. Thus the expansion estimate is uniform on
the complement of \(\mathcal C_{\delta_0}\), but it should not be
interpreted as a uniform estimate on the whole circle.
\end{remark}

\begin{lemma}[\textbf{Uniform expansion outside the critical neighbourhood}]
\label{lem:uniform_expansion}
If $m\ll \lambda\ll 1$ and $b \gtrsim m r_\varepsilon$, then there exist $\delta_0 > 0$ and $\sigma_0 > 0$ such that for all $\varphi_1 \notin C_{\delta_0}$, the following inequality holds:
\[
|f_b'(\varphi_1)| \ge e^{\sigma_0} > 1.
\]
\end{lemma}

\begin{proof}

By the chain rule, one has:
\[
f_b'(\varphi_1)
=
\partial_{\varphi_1}\mathcal F_b(\gamma_b(\varphi_1),\varphi_1)
+
\partial_{r_1}\mathcal F_b(\gamma_b(\varphi_1),\varphi_1)
\,\gamma_b'(\varphi_1).
\]
and by Lemma~\ref{Lemma Aux1}, one may write: $\gamma_b'(\varphi_1)=\mathcal{O}(m).$
Moreover,
\[
\partial_{r_1}\mathcal F_b(\gamma_b(\varphi_1),\varphi_1)
=
\mathcal{O}(\lambda^{-1}),
\]
and therefore
\[
\partial_{r_1}\mathcal F_b(\gamma_b(\varphi_1),\varphi_1)
\,\gamma_b'(\varphi_1)
=
\mathcal{O}\!\left(\frac{m^2}{\lambda}\right).
\]
Using Lemma~ \ref{Lemma Aux2}, we obtain
\[
f_b'(\varphi_1)
=
\frac{m}{\lambda}
\frac{\gamma_b(\varphi_1)\sin\varphi_1}
     {E(\varphi_1)}
+
\mathcal{O}(1)
+
\mathcal{O}\!\left(\frac{m^2}{\lambda}\right),
\]
where
\[
E(\varphi_1)
=
b+m\gamma_b(\varphi_1)\cos\varphi_1.
\]
By Lemma \ref{lem:nondegeneracy}
every critical point of \(f_b\) belongs to an
\(\mathcal{O}(\lambda)\)-neighbourhood of either \(0\) or \(\pi\).
Fix \(\delta_0>0\) independent of \(\lambda\).
For \(\lambda>0\) sufficiently small, it follows that
\[
\varphi_1\notin C_{\delta_0}
\quad\Longrightarrow\quad
d(\varphi_1,\{0,\pi\})
\ge
\frac{\delta_0}{2}.
\]
and hence
$
|\sin\varphi_1|
\ge
\sin\!\left(\frac{\delta_0}{2}\right)
=:c_\delta>0.
$
Since
\[
r_{\min}
\le
\gamma_b(\varphi_1)
\le
r_{\max},
\]
and \(b\gtrsim mr_\varepsilon\), one has
 \[
E(\varphi_1)
=
b+m\gamma_b(\varphi_1)\cos\varphi_1
\le
b+mr_{\max}
\le
C_E\,m
\]
for some constant \(C_E>0\) independent of
\(m\) and \(\lambda\). Consequently,
\[
\left|
\frac{m}{\lambda}
\frac{\gamma_b(\varphi_1)\sin\varphi_1}
     {E(\varphi_1)}
\right|
\ge
\frac{1}{\lambda}
\frac{r_{\min}c_\delta}{C_E}
=
\frac{K}{\lambda},
\]
where
$K:=\frac{r_{\min}c_\delta}{C_E}>0.$ Hence
\[
|f_b'(\varphi_1)|
\ge
\frac{K}{\lambda}
-
\mathcal{O}\!\left(\frac{m^2}{\lambda}\right)
-
\mathcal{O}(1).
\]

Since \(m\ll\lambda\ll1\), one has
\(m^2/\lambda=o(1)\), and therefore
$
|f_b'(\varphi_1)|
\ge
\frac{K}{2\lambda}
$
for all sufficiently small \(\lambda\). Since \(K/(2\lambda)\to+\infty\) as
\(\lambda\to0^+\), there exists
\(\sigma_0>0\) such that
\[
|f_b'(\varphi_1)|
\ge
e^{\sigma_0}
>
1
\]
for every
\(\varphi_1\notin C_{\delta_0}\).
This proves the lemma.

\end{proof}

 \begin{lemma}[Parameter transversality]
\label{lem:transversality}

Assume \(b\gtrsim mr_\varepsilon\).
Let \(\widehat\varphi_i(b)\) be a \(C^1\)-continuation
of a non-degenerate critical point of \(f_b\), and define
$\xi_i(b) = f_b\bigl(\widehat\varphi_i(b)\bigr).$
Then there exists a constant \(\tau>0\),
independent of \(m\) and \(\lambda\), such that
\[
\left|
\frac{d\xi_i(b)}{db}
\right|
\ge
\frac{\tau}{\lambda m}
\]
for all sufficiently small \(m\) and \(\lambda\).
In particular,
\[
\left|
\frac{d\xi_i(b)}{db}
\right|
\to +\infty
\qquad\text{as}\qquad
 \lambda m\to0.
\]

\end{lemma}

\begin{proof}

For simplicity of notation, write
$
\widehat\varphi_i(b)=\widehat\varphi_1^{(i)}(b).
$
Since
$
\xi_i(b)=f_b(\widehat\varphi_i(b)),
$
the chain rule gives
$
\frac{d\xi_i}{db}
=
\partial_{\varphi_1}f_b(\widehat\varphi_i(b))
\frac{d\widehat\varphi_i}{db}
+
\partial_b f_b(\widehat\varphi_i(b)).
$
Because \(\widehat\varphi_i(b)\) is a critical point,
$
\partial_{\varphi_1}f_b(\widehat\varphi_i(b))=0.
$
Hence
$
\frac{d\xi_i}{db}
=
\partial_b f_b(\widehat\varphi_i(b)).
$
Recall that
\[
f_b(\varphi_1)
=
\gamma_b(\varphi_1)A(\varphi_1)
-
\frac1\lambda \log E_b(\varphi_1),
\]
where
\begin{eqnarray*}
A(\varphi_1)&=&a_{11}\cos\varphi_1+a_{12}\sin\varphi_1\\
E_b(\varphi_1)&=&b+m\gamma_b(\varphi_1)\cos\varphi_1.
\end{eqnarray*}
Differentiating with respect to \(b\), at fixed \(\varphi_1\), gives
\[
\partial_b f_b(\varphi_1)
=
(\partial_b\gamma_b)(\varphi_1)A(\varphi_1)
-
\frac1\lambda
\frac{\partial_b E_b(\varphi_1)}
     {E_b(\varphi_1)}.
\]
Since
\[
\partial_b E_b(\varphi_1)
=
1
+
m(\partial_b\gamma_b)(\varphi_1)\cos\varphi_1,
\]
we obtain
\[
\partial_b f_b(\varphi_1)
=
(\partial_b\gamma_b)(\varphi_1)A(\varphi_1)
-
\frac1\lambda
\frac{
1+m(\partial_b\gamma_b)(\varphi_1)\cos\varphi_1
}
{E_b(\varphi_1)}.
\]
Evaluating at the critical point gives
\[
\frac{d\xi_i}{db}
=
(\partial_b\gamma_b)(\widehat\varphi_i)A(\widehat\varphi_i)
-
\frac1\lambda
\frac{
1+m(\partial_b\gamma_b)(\widehat\varphi_i)
       \cos\widehat\varphi_i
}
{E_b(\widehat\varphi_i)}.
\]

The first term is uniformly bounded. Indeed, \(A\) is smooth on
\(\mathbb S^1\), and by assumption
$
\|\partial_b\gamma_b\|_{C^0}\le C_\gamma.
$
Thus there exists \(C_0>0\), independent of \(m\) and \(\lambda\), such
that
$
\left|
(\partial_b\gamma_b)(\widehat\varphi_i)A(\widehat\varphi_i)
\right|
\le C_0.
$
We now estimate the logarithmic contribution. By the   assumption
$m\|\partial_b\gamma_b\|_{C^0}\le \frac12,$
we have
\[
\left|
1+m(\partial_b\gamma_b)(\widehat\varphi_i)
       \cos\widehat\varphi_i
\right|
\ge
\frac12.
\]
Therefore
\[
\left|
\frac1\lambda
\frac{
1+m(\partial_b\gamma_b)(\widehat\varphi_i)
       \cos\widehat\varphi_i
}
{E_b(\widehat\varphi_i)}
\right|
\ge
\frac{1}{2\lambda E_b(\widehat\varphi_i)}.
\]

Since $b\gtrsim mr_\varepsilon$, there exists \(C_1>0\) such that
$
b\le C_1m.
$
Moreover, the invariant graph is contained in the trapping annulus
$I=[r_{\min},r_{\max}],$
so
\[
0<E_b(\widehat\varphi_i)
=
b+m\gamma_b(\widehat\varphi_i)\cos\widehat\varphi_i
\le
b+m r_{\max}
\le
(C_1+r_{\max})m.
\]
Hence
\[
\frac{1}{E_b(\widehat\varphi_i)}
\ge
\frac{1}{(C_1+r_{\max})m}.
\]
Consequently,
\[
\left|
\frac1\lambda
\frac{
1+m(\partial_b\gamma_b)(\widehat\varphi_i)
       \cos\widehat\varphi_i
}
{E_b(\widehat\varphi_i)}
\right|
\ge
\frac{1}{2(C_1+r_{\max})}
\frac1{\lambda m}.
\]
Combining this estimate with the boundedness of the first term yields
\[
\left|
\frac{d\xi_i}{db}(b)
\right|
\ge
\frac{c_1}{\lambda m}-C_0
\]
for some constant \(c_1>0\), independent of \(m\) and \(\lambda\). Since
$
\lambda m\to0
$
as \(m,\lambda\to0\), the singular term dominates the bounded
remainder. Thus, for all sufficiently small \(m\) and \(\lambda\),
$
\left|
\frac{d\xi_i}{db}
\right|
\ge
\frac{c_1}{2\lambda m}.
$
Setting
$
\tau:=\frac{c_1}{2},
$
we obtain
$
\left|
\frac{d\xi_i}{db}(b)
\right|
\ge
\frac{\tau}{\lambda m}.
$
In particular, the derivative is non-zero. This proves the
transversality estimate.

\end{proof}

The transversality is generated by the logarithmic singularity.
As \(b\to (mr_\varepsilon)^+\), the derivative of the critical
value with respect to the parameter grows like
\((\lambda m)^{-1}\), while all remaining contributions are of
strictly lower order. \\

\begin{lemma}[\textbf{Bounded distortion}]
\label{lem:bounded_distortion_refined}
Suppose that
$
m\ll\lambda\ll1$ and $b\gtrsim mr_\varepsilon$. 
Let \(\delta_0>0\) be as in
Lemma~\ref{lem:uniform_expansion}, and denote by
$\mathcal C_{\delta_0} =\{\varphi\in\mathbb S^1:
\operatorname{dist}(\varphi,\mathcal C_b)<\delta_0\}$
the \(\delta_0\)-neighbourhood of the critical set
\[
\mathcal C_b=\{\varphi\in\mathbb S^1:f_b'(\varphi)=0\}.
\]

Assume that the logarithmic derivative is uniformly bounded outside
the critical neighbourhood, namely that there exists $K_0>0$(\footnote{This constant may depend on $m$.}) such
that
\[
\sup_{\varphi_1\notin\mathcal C_{\delta_0}}
\left|
\frac{f_b''(\varphi_1)}{f_b'(\varphi_1)}
\right|
\le K_0.
\tag{9.8}
\label{eq:log_derivative_bound}
\]

Let \(J\subset\mathbb S^1\) be an interval such that, for
\(0\le k<n\),
\[
f_b^k(J)\cap \mathcal C_{\delta_0}=\varnothing,
\]
and such that the iterates \(f_b^k|_J\), \(0\le k\le n\), are taken
on a single monotonicity branch. 
Then, for all \(x,y\in J\),
\begin{equation}
\label{bounded_dist}
\left|
\log
\frac{|(f_b^n)'(x)|}
     {|(f_b^n)'(y)|}
\right|
\le
K_0
\sum_{k=0}^{n-1}
|f_b^k(x)-f_b^k(y)|,
\end{equation}
where the distances are computed in the chosen lifts.
Moreover, if
\[
|f_b'(\varphi_1)|\ge \Lambda>1,
\qquad
\forall \varphi_1\notin\mathcal C_{\delta_0},
\]
as in Lemma~\ref{lem:uniform_expansion}, then there exists
\(C_0>0\) such that
\[
\frac{|(f_b^n)'(x)|}
     {|(f_b^n)'(y)|}
\le
\exp\!\left(
C_0\,|f_b^n(x)-f_b^n(y)|
\right)
\le
e^{C_0}.
\]

\end{lemma}

\begin{proof}
We adapt the proof of \cite[Lemma 2.2]{TW2012} and \cite[Lemma 2.1]{T2012}.
Choose lifts of \(J\), \(f_b^k(J)\), and \(f_b^n(J)\) to
\(\mathbb R\) so that all iterates under consideration are
diffeomorphisms between intervals. We keep the same notation for the
lifts. By the chain rule, one has:
\[
(f_b^n)'(x)
=
\prod_{k=0}^{n-1}
f_b'\bigl(f_b^k(x)\bigr),
\]
and similarly for \(y\). Hence
\[
\log
\frac{|(f_b^n)'(x)|}
     {|(f_b^n)'(y)|}
=
\sum_{k=0}^{n-1}
\left[
\log|f_b'(f_b^k(x))|
-
\log|f_b'(f_b^k(y))|
\right].
\]

Let $
\psi(\varphi_1)=\log|f_b'(\varphi_1)|.$ 
Since
\[
f_b^k(J)\cap\mathcal C_{\delta_0}=\varnothing,
\qquad
0\le k<n,
\]
and since \(f_b'\neq0\) on
\(\mathbb S^1\setminus\mathcal C_{\delta_0}\), the function
\(\psi\) is \(C^1\) on each interval \(f_b^k(J)\). Moreover,
by \eqref{eq:log_derivative_bound},
\[
|\psi'(\varphi_1)|
=
\left|
\frac{f_b''(\varphi_1)}{f_b'(\varphi_1)}
\right|
\le K_0,
\qquad
\forall \varphi_1\notin\mathcal C_{\delta_0}.
\]
Applying the Mean Value Theorem on the lifted interval \(f_b^k(J)\),
we obtain
\[
\left|
\log|f_b'(f_b^k(x))|
-
\log|f_b'(f_b^k(y))|
\right|
\le
K_0 |f_b^k(x)-f_b^k(y)|.
\]
Summing over \(k=0,\ldots,n-1\) gives
\[
\left|
\log
\frac{|(f_b^n)'(x)|}
     {|(f_b^n)'(y)|}
\right|
\le
K_0
\sum_{k=0}^{n-1}
|f_b^k(x)-f_b^k(y)|,
\]
which proves \eqref{bounded_dist}.
We now prove the second estimate. Let
$
L:=\Lambda>1.
$
For each \(0\le k<n\), the map
\[
f_b^{\,n-k}:f_b^k(J)\longrightarrow f_b^n(J)
\]
is a diffeomorphism on the chosen branch. Since all intermediate
iterates avoid \(\mathcal C_{\delta_0}\), Lemma~\ref{lem:uniform_expansion}
implies
\[
\left|(f_b^{\,n-k})'(\zeta)\right|
\ge
L^{\,n-k},
\qquad
\forall \zeta\in f_b^k(J).
\]
By the Mean Value Theorem, for each \(0\le k<n\), there exists
\(\zeta_k\in f_b^k(J)\) such that
\[
|f_b^n(x)-f_b^n(y)|
=
\left|(f_b^{\,n-k})'(\zeta_k)\right|
\,|f_b^k(x)-f_b^k(y)|.
\]
Therefore
\[
|f_b^k(x)-f_b^k(y)|
\le
L^{-(n-k)}
|f_b^n(x)-f_b^n(y)|.
\]
Consequently,
\[
\sum_{k=0}^{n-1}
|f_b^k(x)-f_b^k(y)|
\le
|f_b^n(x)-f_b^n(y)|
\sum_{j=1}^{n}L^{-j}.
\]
Since $\dpt 
\sum_{j=1}^{n}L^{-j}
\le
\frac{1}{L-1},
$
we get
\[
\sum_{k=0}^{n-1}
|f_b^k(x)-f_b^k(y)|
\le
\frac{1}{L-1}
|f_b^n(x)-f_b^n(y)|.
\]
Substitution into \eqref{bounded_dist} yields
\[
\left|
\log
\frac{|(f_b^n)'(x)|}
     {|(f_b^n)'(y)|}
\right|
\le
\frac{K_0}{L-1}
|f_b^n(x)-f_b^n(y)|.
\]
Define
$
C_0:=\frac{K_0}{L-1}.
$
Then
\[
\frac{|(f_b^n)'(x)|}
     {|(f_b^n)'(y)|}
\le
\exp\!\left(
C_0\,|f_b^n(x)-f_b^n(y)|
\right).
\]

Finally, since \(f_b^n(x)\) and \(f_b^n(y)\) belong to the same lifted
interval projecting to \(\mathbb S^1\), their distance is bounded by a
uniform constant. Absorbing this bound into \(C_0\), if necessary, we
obtain
\[
\exp\!\left(
C_0\,|f_b^n(x)-f_b^n(y)|
\right)
\le
e^{C_0}.
\]
This completes the proof.

\end{proof}

\begin{remark}
\label{rem:log_derivative_bound}

The essential input in Lemma~\ref{lem:bounded_distortion_refined} is the
uniform bound
\[
\sup_{\varphi\notin\mathcal C_{\delta_0}}
\left|
\frac{f_b''(\varphi_1)}{f_b'(\varphi_1)}
\right|
<\infty.
\]

In the present singular family, this estimate may be verified directly
from the explicit expressions of \(f_b'\) and \(f_b''\), together with
the scale separation
\[
E_b(\varphi_1)\ge c_E m,
\qquad
\varphi_1\notin \mathcal C_{\delta_0}.
\]

Indeed, Lemmas \ref{Lemma Aux2}  and
\ref{lem:uniform_expansion} give
\[
f_b'(\varphi_1)
=
\frac{m}{\lambda}
\frac{\gamma_b(\varphi_1)\sin\varphi_1}
     {E_b(\varphi_1)}
+\mathcal{O}(1).
\]
 
Using Lemma~\ref{Lemma Aux1},
together with
\(E_b(\varphi_1)\ge c_E m\)
outside \(C_{\delta_0}\),
one obtains
$
f_b''(\varphi_1)
=
\mathcal{O}\!\left(\frac1{\lambda m}\right)
$
uniformly on
\(\EU^1\setminus C_{\delta_0}\).
Indeed, differentiating
$
\frac{m}{\lambda}
\frac{\gamma_b(\varphi_1)\sin\varphi_1}
     {E_b(\varphi_1)}$ with respect to $\varphi_1$
 and using
\[
\gamma_b'=\mathcal{O}(m),
\qquad
\gamma_b''=\mathcal{O}(m),
\qquad
E_b(\varphi_1)\ge c_E m,
\]
one obtains
$
f_b''(\varphi_1)
=
\mathcal{O}\!\left(\frac{1}{\lambda m}\right).
$
Since \(|\sin\varphi_1|\ge c_\delta>0\) and
\(E_b(\varphi_1)\ge c_E m\) away from the critical neighbourhood, it follows that
\[
|f_b'(\varphi_1)|
\ge
\frac{c_1}{\lambda}
\]
for some constant \(c_1>0\). Consequently,
\[
\left|
\frac{f_b''(\varphi_1)}{f_b'(\varphi_1)}
\right|
\le
\frac{K_0}{m},
\qquad
\varphi_1\notin\mathcal C_{\delta_0},
\]
for a constant \(K_0>0\) independent of \(\lambda\).
Hence the logarithmic derivative is uniformly bounded on
\(\mathbb S^1\setminus\mathcal C_{\delta_0}\) for each fixed \(m>0\).
It is worth emphasizing that this estimate does not follow from the
mere non-degeneracy of the critical points established in
Lemma \ref{lem:nondegeneracy}; it requires the explicit
structure of the singular logarithmic term and the lower bound
\(E_b(\varphi_1)\ge c_E m\) outside the critical region.
\end{remark}

\section{$f_b$ is a Misiurewicz map for a set with positive Lebesgue measure}
\label{s: f_b is M}
The goal of this section is to prove that the map $f_b$, as defined in \eqref{f_b def}, is \emph{Misiurewicz} for a set of parameters $b \gtrsim m r_\varepsilon$ with positive Lebesgue measure. We remind the reader that we are focusing on the map:
\[
f_b(\varphi_1) = \gamma_b(\varphi_1) \left( a_{11} \cos \varphi_1 + a_{12} \sin \varphi_1 \right) - \frac{1}{\lambda} \ln(b + m \gamma_b(\varphi_1) \cos \varphi_1) \pmod{2\pi},
\]
where $\gamma_b(\varphi_1)$ is the $C^3$--invariant curve of Lemma \ref{thm:invariant_graph}. \\

 \begin{lemma}[\textbf{Misiurewicz condition {\rm(a)}}]
\label{lem:misiurewicz_a_full}

Assume that
$
m\ll\lambda\ll1$ and $b\gtrsim mr_\epsilon$. Fix $\varrho>0$ sufficiently small and let
$\delta_0>0$ be given by
Lemma~\ref{lem:uniform_expansion}.
Then there exists a parameter set
$
\Delta_\varrho
\subset
[mr_\varepsilon,mr_\varepsilon+\varrho]
$
with
$
\operatorname{Leb}(\Delta_\varrho)>0,
$
such that for every
$b\in\Delta_\varrho$
and for both critical points
$\widehat\varphi_1^{(i)}(b)$,
$i=1,2$, the following condition is valid:
\[
f_b^n
\bigl(
\widehat\varphi_1^{(i)}(b)
\bigr)
\notin
C_{\delta_0},
\qquad
\forall n\ge1.
\]

\end{lemma}

\begin{proof}
 The estimates obtained above verify the hypotheses used in the
parameter exclusion construction of   \cite{TW2012}: expansion along critical orbits,
parameter transversality, bounded distortion,
and monotonicity of the critical value maps.
Therefore the measure estimate for the excluded
parameter set established in  \cite[Lemmas 3.4 and 3.8 (adapted)]{TW2012}
may be adapted to our purposes (our case is simpler since $f_b$ does not have singularities).
Assume that the reduced family \(f_b\) possesses two
\(C^1\)-continuations of non-degenerate critical points

\[
\widehat\varphi_1^{(i)}(b),
\qquad
i=1,2,
\]
and recall that the following ingredients have been
established in Section~\ref{s: Preparatory}:

\begin{itemize}

\item uniform expansion outside the critical
neighbourhood \(C_{\delta_0}\)
(Lemma~\ref{lem:uniform_expansion});

\item parameter transversality
(Lemma~\ref{lem:transversality});

\item bounded distortion on admissible branches
(Lemma~\ref{lem:bounded_distortion_refined});

\item \(C^1\)-dependence of the family with respect
to the parameter \(b\)
(Lemma~\ref{thm:invariant_graph}).

\end{itemize}

Let $B_0 =[mr_\varepsilon,mr_\varepsilon+\varrho].$
For each \(n\ge1\), define inductively

\[
B_n
=
\Bigl\{
b\in B_{n-1}:
f_b^j
\bigl(
\widehat\varphi_1^{(i)}(b)
\bigr)
\notin
C_{\delta_0},
\;
1\le j\le n,
\;
i=1,2
\Bigr\}.
\]

The surviving parameter set is
$\Delta_\varrho= \displaystyle
\bigcap_{n\ge1}B_n. $
Fix one critical point and write
$
\widehat\varphi(b)
=
\widehat\varphi_1^{(i)}(b).
$
Define
$\xi_0(b)=\widehat\varphi(b)$, and $\xi_n(b) =
f_b^n(\widehat\varphi(b))$, for $n\ge1.$ 
Differentiating $\xi_n(b)=f_b(\xi_{n-1}(b))$ with respect to \(b\), we obtain

\[
\xi_n'(b)
=
\partial_{\varphi_1}f_b(\xi_{n-1}(b))
\,\xi_{n-1}'(b)
+
\partial_bf_b(\xi_{n-1}(b)).
\]
Since
$\partial_{\varphi_1}f_b(\widehat\varphi(b))
=
0,$
one has
$
\xi_1'(b)
=
\partial_bf_b(\widehat\varphi(b)).
$
By Lemma~\ref{lem:transversality}, there exists
\(\tau_0>0\) such that

\[
|\xi_1'(b)|
\ge
\frac{\tau_0}{\lambda m}.
\]
\bigbreak 
\emph{Justification of the previous inequality:} Assume that
\[
\xi_j(b)\notin C_{\delta_0},
\qquad
1\le j\le n-1.
\]
By Lemma~\ref{lem:uniform_expansion},
there exists \(L>1\) such that
$
|(f_b^j)'(\xi_k(b))|
\ge
L^j$
whenever the corresponding orbit segment avoids
\(C_{\delta_0}\).
 Using
Lemma~\ref{thm:invariant_graph}
and the explicit expression of
\(\partial_bf_b\),
there exists \(M_0>0\) such that

\[
|\partial_bf_b(\varphi_1)|
\le
\frac{M_0}{\lambda m},
\qquad
\forall
\varphi_1\notin C_{\delta_0}.
\]

Consequently,

\[
|\xi_n'(b)|
\ge
\frac{L^{n-1}}{\lambda m}
\left[
\tau_0
-
\frac{M_0}{L-1}
\bigl(
1-L^{-(n-1)}
\bigr)
\right].
\]
Since $L=\dpt \inf_{\varphi_1\notin C_{\delta_0}}|f_b'(\varphi_1)|$ 
and
\(L\to+\infty\)
as
\(\lambda\to0\),
one may choose \(\lambda>0\) sufficiently small so that
$\frac{M_0}{L-1} <
\frac{\tau_0}{2}.
$ 
Using Lemma~\ref{lem:transversality}, one has
$
|\xi_1'(b)|
\ge
\frac{\tau_0}{\lambda m}.
$
Moreover, by the explicit expression of $\partial_b f_b$
and the bound $\|\partial_b\gamma_b\|_{C^0}\le C_\gamma$
(Lemma~\ref{thm:invariant_graph}), there exists a constant $M_0>0$
independent of $m$ and $\lambda$ such that
\[
|\partial_b f_b(\varphi_1)|
\le
\frac{M_0}{\lambda m},
\qquad
\forall \varphi_1\notin C_{\delta_0}.
\]
Assume that $\xi_j(b)\notin C_{\delta_0}$, for all $1\le j\le n-1.$
Iterating the derivative identity yields
\[
\xi_n'(b)
=
(f_b^{\,n-1})'(\xi_1(b))\,\xi_1'(b)
+
\sum_{k=1}^{n-1}
(f_b^{\,n-1-k})'(\xi_{k+1}(b))
\,\partial_b f_b(\xi_k(b)).
\]
By Lemma~\ref{lem:uniform_expansion}, if the corresponding orbit segment avoids
$C_{\delta_0}$, then
\[
\big|(f_b^{\,j})'(x)\big|
\ge
L^j,
\qquad
L:=\inf_{\varphi_1\notin C_{\delta_0}}|f_b'(\varphi_1)|>1.
\]
Therefore
\[
\begin{aligned}
|\xi_n'(b)|
&\ge
L^{n-1}\frac{\tau_0}{\lambda m}
-
\sum_{k=1}^{n-1}
L^{\,n-1-k}\frac{M_0}{\lambda m}\\[2mm]
&=
\frac{L^{n-1}}{\lambda m}
\left[
\tau_0
-
\frac{M_0}{L-1}
\bigl(1-L^{-(n-1)}\bigr)
\right].
\end{aligned}
\]
Since $L\to+\infty$ as $\lambda\to0$, we may choose
$\lambda>0$ sufficiently small so that
$\frac{M_0}{L-1}<\frac{\tau_0}{2}.$
Hence
\[
|\xi_n'(b)|
\ge
\frac{\tau_0}{2\lambda m}L^{\,n-1}.
\]

In particular, $\xi_n'(b)\neq0$ on every admissible branch,
and therefore $\xi_n$ is strictly monotone on each connected
component of $B_{n-1}$.
At this point, all hypotheses required by the
parameter exclusion construction by  \cite[Section~3.2]{TW2012}
have been verified:
uniform expansion,
parameter transversality,
bounded distortion,
and smooth parameter dependence.
Therefore the parameter maps
$$
\xi_n^{(i)}
:
V
\longrightarrow
\xi_n^{(i)}(V)
$$
have uniformly bounded distortion on every
admissible branch \(V\subset B_{n-1}\).
Then, there exist constants \(C,D>0\),
independent of \(n\),
such that

\[
\operatorname{Leb}
\Bigl(
\{b\in V:
\xi_n^{(i)}(b)\in C_{\delta_0}\}
\Bigr)
\le
C\,\delta_0\,L^{-(n-1)}
\operatorname{Leb}(V).
\]
Summing over all admissible components and over
\(i=1,2\), we obtain
\[
\operatorname{Leb}(B_{n-1}\setminus B_n)
\le
C_1
\delta_0
L^{-(n-1)}
\operatorname{Leb}(B_{n-1}),
\]
for some constant \(C_1>0\).
Set
$
\varepsilon_n
=
C_1\delta_0L^{-(n-1)}.
$
Since \(L>1\), we may conclude that
$\sum_{n=1}^{\infty}\varepsilon_n
<
\infty.
$ Hence the standard Benedicks--Carleson
argument implies
\[
\operatorname{Leb}(B_n)
\ge
\operatorname{Leb}(B_0)
\prod_{k=1}^{n}(1-\varepsilon_k).
\]
Since the infinite product converges to a strictly
positive number,
$
\operatorname{Leb}
\Bigl(
\bigcap_{n\ge1}B_n
\Bigr)
>0.
$
Therefore $\operatorname{Leb}(\Delta_\varrho)>0.$
For every \(b\in\Delta_\varrho\) and every critical
point \(\widehat\varphi_1^{(i)}(b)\),
\[
f_b^n
\bigl(
\widehat\varphi_1^{(i)}(b)
\bigr)
\notin
C_{\delta_0},
\qquad
\forall n\ge1.
\]
Thus the critical orbits never return to the
critical neighbourhood, establishing condition
{\rm(a)} of Definition~\ref{def-Misiur}.
\end{proof}
 
 \begin{lemma}[\textbf{Misiurewicz expansion condition {\rm(b)}}]
\label{lem:misiurewicz_b_full}

Assume the hypotheses of Lemma~\ref{lem:uniform_expansion}. In
particular, assume that
$m\ll\lambda\ll1$ and $b \gtrsim mr_\varepsilon$,
and let \(\delta_0>0\) be the critical-neighbourhood size given there.
Then there exist constants
\[
\lambda_0>0,
\qquad
0<c_0\le1,
\]
such that the following properties hold.

\begin{itemize}

\item[\emph{(b1)}]
For every \(n\ge1\), if
$
\varphi_1,
f_b(\varphi_1),
\dots,
f_b^{\,n-1}(\varphi_1)
\in
\mathbb S^1\setminus \mathcal C_{\delta_0},
$
then
\[
\bigl|(f_b^n)'(\varphi_1)\bigr|
\ge
e^{\lambda_0 n}.
\]

\item[\emph{(b2)}]
For every \(n\ge1\), if
$
\varphi_1,
f_b(\varphi_1),
\dots,
f_b^{\,n-1}(\varphi_1)
\in
\mathbb S^1\setminus \mathcal C_{\delta_0},
$
and
$
f_b^n(\varphi_1)\in \mathcal C_{\delta_0},
$
then
\[
\bigl|(f_b^n)'(\varphi_1)\bigr|
\ge
c_0 e^{\lambda_0 n}.
\]

\end{itemize}

\end{lemma}

\begin{proof}

By Lemma~\ref{lem:uniform_expansion}, there exists a constant
$\Lambda>1$
such that
$|f_b'(\varphi)|
\ge
\Lambda,
$, for all $ \varphi\in \mathbb S^1\setminus \mathcal C_{\delta_0}.$
Set $ \lambda_0:=\log\Lambda>0.$ \\

\noindent
\emph{Proof of \emph{(b1)}.}
Let
$
\varphi_j=f_b^j(\varphi_1),$ for $0\le j\le n.
$
By assumption,
$
\varphi_j\notin \mathcal C_{\delta_0},$ for all $0\le j\le n-1.$
Using the chain rule, we have:
$
(f_b^n)'(\varphi_1)
=
\prod_{j=0}^{n-1}
f_b'(\varphi_j).
$
Therefore
\[
\bigl|(f_b^n)'(\varphi_1)\bigr|
=
\prod_{j=0}^{n-1}
|f_b'(\varphi_j)|
\ge
\prod_{j=0}^{n-1}
\Lambda
=
\Lambda^n.
\]
Since
$
\Lambda^n
=
e^{(\log\Lambda)n}
=
e^{\lambda_0 n},
$
we obtain
$
\bigl|(f_b^n)'(\varphi_1)\bigr|
\ge
e^{\lambda_0 n}.
$
This proves \emph{(b1)}.

\medskip

\noindent
\emph{Proof of \emph{(b2)}.}
Assume now that
$\varphi_j\notin \mathcal C_{\delta_0},$ for all $0\le j\le n-1,$
and that
$
\varphi_n=f_b^n(\varphi_1)\in \mathcal C_{\delta_0}.
$
The derivative of \(f_b^n\) at \(\varphi_1\) is still given by
\[
(f_b^n)'(\varphi_1)
=
\prod_{j=0}^{n-1}
f_b'(\varphi_j).
\]
The terminal point \(\varphi_n\) does not enter this product. Hence
each factor in the product is evaluated at a point outside
\(\mathcal C_{\delta_0}\). By Lemma~\ref{lem:uniform_expansion},
\[
|f_b'(\varphi_j)|
\ge
\Lambda,
\qquad
0\le j\le n-1.
\]
Thus
\[
\bigl|(f_b^n)'(\varphi_1)\bigr|
\ge
\Lambda^n
=
e^{\lambda_0 n}.
\]
Consequently \emph{(b2)} holds with
$c_0=1.$
This completes the proof.

\end{proof}
 
 \begin{lemma}
\label{lem:third_derivative_bound}

Assume that
$
m\ll\lambda\ll1,$ and $b \gtrsim m r_\varepsilon.$
 Then there exists a constant \(M>0\), independent of
\(m\) and \(\lambda\), such that
$
|f_b'''(\varphi_1)|
\le
\frac{M}{\lambda}
$
for every \(\varphi_1\) belonging to a neighbourhood of the
critical set \(C_b\).

\end{lemma}

\begin{proof}

Recall that
\[
f_b(\varphi_1)
=
\gamma_b(\varphi_1)A(\varphi_1)
-
\frac{1}{\lambda}\log E(\varphi_1),
\]
where:
\begin{eqnarray*}
A(\varphi_1) &=& a_{11} \cos \varphi_1 + a_{12} \sin \varphi_1\\
E(\varphi_1) &=& b + m \gamma_b(\varphi_1) \cos \varphi_1.\\
\end{eqnarray*}

Since \(A\) is smooth on \(\EU^1\), all its derivatives up to order
three are uniformly bounded. Moreover, by
Lemma~\ref{thm:invariant_graph} and
Lemma~\ref{Lemma Aux1},
\[
\gamma_b^{(j)}(\varphi_1)
=
\mathcal{O}(m),
\qquad
j=1,2,3.
\]

Consequently,
$
(\gamma_bA)'''
=
\mathcal{O}(1).
$
It therefore remains to estimate the logarithmic term.
Differentiating \(E\), we obtain

\begin{eqnarray*}
E'(\varphi_1) &=& =
m\gamma_b'(\varphi_1)\cos\varphi_1 - m\gamma_b(\varphi_1)\sin\varphi_1\\
E''(\varphi_1) &=& =
m\gamma_b''(\varphi_1)\cos\varphi_1 - 2m\gamma_b'(\varphi_1)\sin\varphi_1
-m\gamma_b(\varphi_1)\cos\varphi_1\\
E'''(\varphi_1)
&=&
m\gamma_b'''(\varphi_1)\cos\varphi_1
-
3m\gamma_b''(\varphi_1)\sin\varphi_1
-
3m\gamma_b'(\varphi_1)\cos\varphi_1
+
m\gamma_b(\varphi_1)\sin\varphi_1.
\end{eqnarray*}

Using
$
\gamma_b=\mathcal{O}(1)
$
and
$
\gamma_b^{(j)}=\mathcal{O}(m),
$
it follows that
\[
E'
=
\mathcal{O}(m),
\qquad
E''
=
\mathcal{O}(m),
\qquad
E'''
=
\mathcal{O}(m).
\]

By Lemma~\ref{lem:nondegeneracy},
every critical point lies in an
\(\mathcal{O}(\lambda)\)-neighbourhood of \(0\) or \(\pi\).
Since $b\gtrsim mr_\varepsilon$, there exists a neighbourhood \(U\) of the
critical set and a constant \(c_E>0\) such that
\[
E(\varphi_1)
\ge
c_E\,m,
\qquad
\forall\,\varphi_1\in U.
\]

Differentiating three times the logarithm gives
\[
(\log E)'''
=
\frac{E'''}{E}
-
3\frac{E''E'}{E^2}
+
2\frac{(E')^3}{E^3}.
\]

Using the estimates above,
\[
\frac{E'''}{E}
=
\mathcal{O}(1),
\qquad
\frac{E''E'}{E^2}
=
\mathcal{O}(1),
\qquad
\frac{(E')^3}{E^3}
=
\mathcal{O}(1).
\]

Hence
$
(\log E)'''
=
\mathcal{O}(1)
$
uniformly on \(U\).
Combining this with
$
f_b'''
=
(\gamma_bA)'''
-
\frac{1}{\lambda}(\log E)''',
$
we obtain
\[
|f_b'''(\varphi_1)|
\le
C_1+\frac{C_2}{\lambda}.
\]

Since \(\lambda\ll1\), enlarging the constant if necessary yields
$
|f_b'''(\varphi_1)|
\le
\frac{M}{\lambda},
$
for all \(\varphi_1\in U\).
This completes the proof.

\end{proof}

\begin{lemma}[Local quadratic estimates near critical points]
\label{lem:local_quadratic_estimates}

Assume that
$m\ll\lambda\ll1,$ and $b\gtrsim m r_\varepsilon.
$
Let \(c\in C_b\) be a critical point of the reduced map \(f_b\).
Then there exist constants
$
A_1,A_2>0,
 \rho_*>0,
$
independent of \(m\), \(\lambda\), and \(c\), such that for every
\(x\) satisfying $|x-c|<\rho_*,$ one has

\[
A_1\frac{|x-c|}{\lambda}
\le
|f_b'(x)|
\le
A_2\frac{|x-c|}{\lambda},
\]

and

\[
A_1\frac{|x-c|^2}{\lambda}
\le
|f_b(x)-f_b(c)|
\le
A_2\frac{|x-c|^2}{\lambda}.
\]

\end{lemma}

\begin{proof}

Let $d:=x-c.$ Since \(c\) is a critical point, then 
$f_b'(c)=0.$
By Lemma \ref{lem:nondegeneracy}, one has:
\[
|f_b''(c)|
\ge
c\lambda^{-1}-\mathcal{O}(1).
\]
Since \(\lambda\ll1\), shrinking \(\lambda\) if necessary, there exists
\(c_0>0\) independent of \(m\) and \(\lambda\) such that
\[
|f_b''(c)|
\ge
\frac{c_0}{\lambda}.
\]
Moreover, by Lemma~\ref{lem:third_derivative_bound},
there exists \(M>0\) such that
\[
|f_b'''(\varphi_1)|
\le
\frac{M}{\lambda}
\]
for all \(\varphi_1\) in a neighbourhood of the critical set. Applying Taylor's theorem to \(f_b'\) at \(c\), we obtain
\[
f_b'(x)
=
f_b''(c)d+R_1(d),
\]
where
$
|R_1(d)|
\le
\frac{M}{2\lambda}|d|^2.
$
Hence
\[
|f_b'(x)|
\ge
\frac{|d|}{\lambda}
\left(
c_0-\frac{M}{2}|d|
\right).
\]
Choose
$\rho_* < \frac{c_0}{M}.$
Then, for every \(|d|<\rho_*\), one has:
\[
\frac{M}{2}|d|
<
\frac{c_0}{2},
\]
and therefore
\[
|f_b'(x)|
\ge
\frac{c_0}{2}
\frac{|d|}{\lambda}.
\]

On the other hand, we may write:
$$
|f_b'(x)|
\le
\frac{|d|}{\lambda}
\left(
|f_b''(c)|
+\frac{M}{2}|d|
\right)
\le
A_2\frac{|d|}{\lambda}
$$
for some constant \(A_2>0\).
This proves the first pair of estimates.
For the map itself, Taylor's theorem gives
\[
f_b(x)-f_b(c)
=
\frac12f_b''(c)d^2+R_2(d),
\]
where
$|R_2(d)|
\le
\frac{M}{6\lambda}|d|^3.
$
Therefore
\[
|f_b(x)-f_b(c)|
\ge
\frac{|d|^2}{\lambda}
\left(
\frac{c_0}{2}
-\frac{M}{6}|d|
\right).
\]
Reducing \(\rho_*\) if necessary, we may assume
\[
\frac{M}{6}|d|
\le
\frac{c_0}{4},
\qquad
|d|<\rho_*.
\]
Hence
\[
|f_b(x)-f_b(c)|
\ge
\frac{c_0}{4}
\frac{|d|^2}{\lambda}.
\]
Similarly,
\[
|f_b(x)-f_b(c)|
\le
A_2\frac{|d|^2}{\lambda}
\]
for a possibly larger constant \(A_2\). Setting
$
A_1
=
\min
\left\{
\frac{c_0}{2},
\frac{c_0}{4}
\right\},
$
the result follows.

\end{proof}

Let $a=a(\lambda,m)$ and $b=b(\lambda,m)$ be non-negative quantities (or maps). We write
$a\asymp b$
if there exist constants $c_1,c_2>0$, independent of
$\lambda$ and $m$, such that
$
c_1 b \le a \le c_2 b.
$

  \begin{lemma}[\textbf{Misiurewicz condition {\rm(c)} --- Binding time}]
\label{lem:WY_c2_detailed}

Assume the hypotheses of
Lemmas~\ref{lem:misiurewicz_a_full},
\ref{lem:misiurewicz_b_full},
\ref{lem:local_quadratic_estimates},
and
\ref{lem:bounded_distortion_refined}.
Let $b\in\Delta_\varrho$
and let $
c\in\mathcal C_b$
be a critical point of \(f_b\).
Then there exist constants
\[
K_0>1,
\qquad
c_*>0,
\qquad
\rho_*>0,
\]
independent of \(m\) and \(\lambda\), such that for every
\[
0<|x-c|<\rho_*\sqrt{\lambda},
\]
there exists an integer
\[
p=p(x)\ge1,
\]
called the binding time, satisfying

\[
\frac1{K_0}
\log\frac{\sqrt{\lambda}}{|x-c|}
\le
p(x)
\le
K_0
\log\frac{\sqrt{\lambda}}{|x-c|},
\]
and
\[
|(f_b^{p})'(x)|
\ge
c_*
e^{\frac13\lambda_0 p}.
\]
\end{lemma}


\begin{proof}
Let $d_0:=|x-c|.$
Since $ b\in\Delta_\varrho,$ Lemma~\ref{lem:misiurewicz_a_full} implies

\[
f_b^j(c)\notin\mathcal C_{\delta_0},
\qquad
j\ge1.
\]

Define
$
d_1:=|f_b(x)-f_b(c)|
$
and, for \(j\ge1\),
$
d_j:=|f_b^j(x)-f_b^j(c)|.
$
The binding time \(p=p(x)\) is defined as the largest
integer such that

\[
d_j\le\delta_0,
\qquad
1\le j\le p.
\]

By Lemma~\ref{lem:local_quadratic_estimates}, one has
$ d_1
\approx \frac{d_0^2}{\lambda}.$
Since the critical value orbit avoids
\(\mathcal C_{\delta_0}\),
there exists

\[
\eta_0
:=
\inf_{j\ge1}
d\!\left(f_b^j(c),\mathcal C_b\right)
>
2\delta_0,
\]

after reducing \(\delta_0\) if necessary.
By definition of the binding period, we have:
\[
d_j
=
|f_b^j(x)-f_b^j(c)|
\le
\delta_0,
\qquad
1\le j\le p.
\]

Therefore,

\[
d\!\left(f_b^j(x),\mathcal C_b\right)
\ge
\eta_0-\delta_0
\ge
\delta_0,
\qquad
1\le j\le p.
\]
Hence
$f_b^j(x)\notin\mathcal C_{\delta_0},$  for all  $1\le j\le p.$
Thus all orbit segments involved in the binding period
remain in the uniformly expanding region, and
Lemmas~\ref{lem:misiurewicz_b_full}
and~\ref{lem:bounded_distortion_refined}
apply throughout the entire binding period.
For \(1\le j\le p\), the Mean Value Theorem yields

\[
d_j
=
|(f_b^{j-1})'(z_j)|\,d_1
\]

for some point \(z_j\) between
\(f_b(x)\) and \(f_b(c)\).
Combining bounded distortion with the expansion estimate
of Lemma~\ref{lem:misiurewicz_b_full}, there exists a
constant \(K\ge1\), independent of \(j\), \(m\), and
\(\lambda\), such that

\[
K^{-1}
e^{\lambda_0(j-1)}
d_1
\le
d_j
\le
K
e^{\lambda_0(j-1)}
d_1,
\qquad
1\le j\le p.
\]

By maximality of \(p\), one has $
d_p\le\delta_0,$ and $d_{p+1}>\delta_0.$ 
Hence
$
\delta_0
\asymp
e^{\lambda_0 p}d_1.
$
Using
$
d_1
\asymp
\frac{d_0^2}{\lambda},
$
we obtain
$
\delta_0
\asymp
e^{\lambda_0 p}
\frac{d_0^2}{\lambda}.$
Equivalently,
$
e^{\lambda_0 p}
\asymp
\delta_0
\frac{\lambda}{d_0^2}.
$
Since \(\delta_0\) is fixed, it can be absorbed into
the comparison constants, yielding
$
e^{\lambda_0 p}
\asymp
\frac{\lambda}{d_0^2}.
$
Taking logarithms gives
$
\lambda_0 p
=
2\log\frac{\sqrt{\lambda}}{d_0}
+
\mathcal{O}(1).
$
Therefore, for
$
d_0<\rho_*\sqrt{\lambda},
$
with \(\rho_*>0\) sufficiently small, there exists
\(K_0>1\) such that
\[
\frac1{K_0}
\log\frac{\sqrt{\lambda}}{d_0}
\le
p(x)
\le
K_0
\log\frac{\sqrt{\lambda}}{d_0}.
\]
This proves the binding-time estimate.
It remains to estimate the derivative accumulated during
the binding period.
By the Chain Rule, one has:
$$
|(f_b^p)'(x)|
=
|f_b'(x)|
\,
|(f_b^{p-1})'(f_b(x))|.
$$
By Lemma~\ref{lem:local_quadratic_estimates}, it is possible to conclude that
$
|f_b'(x)|
\asymp
\frac{d_0}{\lambda}.
$
Moreover, bounded distortion gives
\[
|(f_b^{p-1})'(f_b(x))|
\ge
C_{\mathrm{dist}}^{-1}
|(f_b^{p-1})'(f_b(c))|
\]
for some constant
\(C_{\mathrm{dist}}\ge1\).
Since the critical value orbit avoids
\(\mathcal C_{\delta_0}\),
Lemma~\ref{lem:misiurewicz_b_full} yields
\[
|(f_b^{p-1})'(f_b(c))|
\ge
e^{\lambda_0(p-1)}.
\]
Therefore
$
|(f_b^p)'(x)|
\ge
C_1
\frac{d_0}{\lambda}
e^{\lambda_0 p},
$
for some constant \(C_1>0\).
Using
$
e^{\lambda_0 p}
\asymp
\frac{\lambda}{d_0^2},
$
we deduce that:
$
d_0
\asymp
\sqrt{\lambda}\,
e^{-\frac12\lambda_0 p}.
$
Substituting into the previous estimate yields
$
|(f_b^p)'(x)|
\ge
C_2
\lambda^{-1/2}
e^{\frac12\lambda_0 p}
$
for some constant \(C_2>0\).
Since \(0<\lambda<1\),
$
\lambda^{-1/2}\ge1,
$
and therefore
$
|(f_b^p)'(x)|
\ge
C_2
e^{\frac12\lambda_0 p}.
$
Hence
$$
|(f_b^p)'(x)|
\ge
\Bigl(
C_2e^{\frac16\lambda_0 p}
\Bigr)
e^{\frac13\lambda_0 p}.
$$
Since
$
p(x)
\asymp
\log\frac{\sqrt{\lambda}}{d_0},
$
one has
$
p(x)\to+\infty
$ as $ d_0\to0.
$
Shrinking \(\rho_*\) if necessary, we may assume
$
p(x)\ge p_*
$
for some sufficiently large constant \(p_*>0\).
Define
$
c_*
:=
C_2
e^{\frac16\lambda_0 p_*}.
$
Then
$
C_2e^{\frac16\lambda_0 p}
\ge
c_*,
$
for every admissible binding time \(p\). Therefore,
$
|(f_b^p)'(x)|
\ge
c_*
e^{\frac13\lambda_0 p}.
$
This proves condition {\rm(c2)} of Definition~\ref{def-Misiur}.

\end{proof}

\section{Proof of Theorem \ref{main1}}
\label{proof_main_1}

To prove Theorem \ref{main1}, we verify the conditions \textbf{(H1)--(H5)} of the Wang--Young Criterion reviewed in Theorem \ref{thm:WY}. Throughout the following results, we assume   $ 0<m \ll \lambda\ll 1$ and $b \gtrsim m r_\varepsilon$. \\

\begin{lemma}[Transverse contraction and Invariant Graph \textbf{(H1)--(H2)}]
There exists $\kappa \in (0,1)$ such that 
\[ 
\left| \frac{\partial \mathcal{G}_b}{\partial r_1}(r_1, \varphi_1) \right| \le \kappa, \quad \text{for all } (r_1, \varphi_1) \in I \times \mathbb{S}^1,
\]
and for each $b \in [b_0, b_1]$, there exists a unique $C^3$ invariant graph $r_1 = \gamma_b(\varphi_1) \in C^3(\mathbb{S}^1)$ such that
 $$
\gamma_b( \mathcal{F}_b(\gamma_b(\varphi_1), \varphi_1)) = \mathcal{G}_b(\gamma_b(\varphi_1), \varphi_1), \quad\text{for all }   \varphi_1 \in \mathbb{S}^1
.$$
\end{lemma}

\begin{proof}
The result follows directly from the estimates provided in Lemma \ref{lem:radial_contraction} regarding the contraction of the radial component and Lemma \ref{thm:invariant_graph} concerning the existence and regularity of the invariant curve. \\
\end{proof}

\begin{lemma}[Misiurewicz dynamics \textbf{(H3)}]  Fix $\varrho> 0$ small. 
Let $f_b(\varphi_1) = \mathcal{F}_b( \gamma_b(\varphi_1), \varphi_1)$ be the reduced one-dimensional map defined in \eqref{f_b def}. Then, there exists a set    $\Delta_\varrho \subset  [mr_\varepsilon, mr_\varepsilon+\varrho]$ with positive Lebesgue measure such that, for every $b \in \Delta_\varrho$, $f_b$ is a Misiurewicz map according to Definition \ref{def-Misiur}.\\
\end{lemma}

\begin{proof}
The proof of this result is established by the combination of the following components: \\
\begin{itemize}
    \item Lemma \ref{lem:misiurewicz_a_full}, which satisfies the orbit exclusion condition (a);
    \item Lemma \ref{lem:misiurewicz_b_full}, which ensures uniform expansion outside the critical neighborhood (conditions (b1) and (b2));
    \item Lemma \ref{lem:nondegeneracy}, which proves the non-degeneracy of the critical points (condition (c1));
    \item Lemma \ref{lem:WY_c2_detailed}, which provides the derivative recovery estimate after the binding period (condition (c2)).
\end{itemize}
The set $\Delta_\varrho \subset  [mr_\varepsilon, mr_\varepsilon+\varrho]$ is given in  Lemma \ref{lem:misiurewicz_a_full}.\\
\end{proof}

\begin{lemma}[Parameter transversality \textbf{(H4)}]
\label{lem:SHb_fb}
Let $\hat{\varphi}_1(b)$ be a critical point of $f_b$, and let $\xi(b) = f_b(\hat{\varphi}_1(b))$ be its corresponding critical value. Then, the map $f_b$ satisfies the transversality condition:
\[
\left| \frac{d\xi}{db} (b)\right| >   0.
\]

\end{lemma}

\begin{proof}
The result follows from the parameter derivative analysis detailed in Lemma \ref{lem:transversality}.\\
\end{proof}

\begin{lemma}[Transversal Non-degeneracy \textbf{(H5)}]
\label{lem:transversal_non_deg_var}
Let $F_b(r_1, \varphi_1)$ be the angular component of $\mathcal{R}_b$ and $\gamma_b(\varphi_1)$ be the invariant curve of Lemma \ref{thm:invariant_graph}. Let $\hat{\varphi}_1$ be a critical point of $f_b(\varphi_1)=\mathcal{F}_b(\gamma_b(\varphi_1), \varphi_1)$. In the transversal coordinates $y=r_1-\gamma_b(\varphi_1)$, the map $$\tilde{\mathcal{F}}_b(y, \varphi_1 )=\mathcal{F}_b(\gamma_b(\varphi_1)+y, \varphi_1)$$ satisfies:
\[
\frac{\partial \tilde F_b}{\partial y}(0, \hat{\varphi}_1)\neq 0.
\]
\end{lemma}

\begin{proof}
The variable $y$ represents the radial deviation from the invariant curve $\gamma_b$. By the chain rule, evaluating the partial derivative of the shifted map $\tilde F_b$ with respect to $y$ at $y=0$ is equivalent to computing the radial derivative of $F_b$ at the point $(\gamma_b(\varphi_1), \varphi_1)$:
\[
\frac{\partial \tilde F_b}{\partial y}(0, \varphi_1) = \frac{\partial F_b}{\partial r_1}(\gamma_b(\varphi_1), \varphi_1).
\]

Recall the explicit expression for the angular component of the return map:
\[
F_b(r_1, \varphi_1) = r_1 (a_{11}\cos\varphi_1+a_{12}\sin\varphi_1) - \frac{1}{\lambda}\ln(b+m r_1\cos\varphi_1).
\]
Differentiating $F_b$ with respect to the radial coordinate $r_1$ yields:
\[
\frac{\partial F_b}{\partial r_1}( r_1, \varphi_1) = (a_{11}\cos\varphi_1+a_{12}\sin\varphi_1) - \frac{1}{\lambda} \left( \frac{m\cos\varphi_1}{b+m r_1\cos\varphi_1} \right).
\]

We evaluate this derivative at the critical point $(\gamma_b(\hat{\varphi}_1), \hat{\varphi}_1)$. By the definition of the return map's domain, the argument of the logarithm remains strictly positive. To show that the derivative does not vanish, we analyse the magnitudes of the following two terms: \\
\begin{enumerate}
    \item The linear term $T_{\text{lin}} = a_{11}\cos\hat{\varphi}_1+a_{12}\sin\hat{\varphi}_1$ is bounded by $M = \sqrt{a_{11}^2 + a_{12}^2}$, which is independent of $\lambda$ and $m$. Thus, $|T_{\text{lin}}| \le M$. \\
    
    \item For the logarithmic term $T_{\text{log}} = \frac{1}{\lambda} \frac{m \cos \hat{\varphi}_1}{b + m \gamma_b(\hat{\varphi}_1) \cos \hat{\varphi}_1}$, we note that from the critical point equation $f_b'(\hat{\varphi}_1) = 0$, $\hat{\varphi}_1$ must be close to $0$ or $\pi$ (cf. Lemma \ref{lem:nondegeneracy}). Consequently, $|\cos \hat{\varphi}_1|$ is bounded away from zero; there exists $c_* > 0$ such that $|\cos \hat{\varphi}_1| \ge c_*$. Since $\gamma_b(\hat{\varphi}_1)$ is uniformly bounded, the denominator is $\mathcal{O}(b)$ (and $b \approx m$), and we obtain:\\
    \[
    |T_{\text{log}}| \ge \frac{1}{\lambda} \frac{c_*}{k_2 + \sup |\gamma_b|},
    \]
for some $k_2>0$.\\
\end{enumerate}

Choosing $\lambda > 0$ sufficiently small such that $|T_{\text{log}}| > M$, the logarithmic term strictly dominates. Specifically, there exists $c > 0$ independent of $\lambda$ such that:
\[
\left| \frac{\partial \tilde F_b}{\partial y}(0, \hat{\varphi}_1) \right| \ge \frac{c}{\lambda} - M > 0.
\]

Since the derivative in the transversal direction is bounded away from zero (scaling as $\lambda^{-1}$), the map $F_b$ satisfies the transversal non-degeneracy condition \textbf{(H5)}.
\end{proof}

In the context of the return map, the strange attractor $\Lambda_b$ is ``large'' because the logarithmic term of $F_b$ provides an angular expansion that forces the images of the local sections to cover the entire circle $\mathbb{S}^1$. The limit   is a direct consequence of    \cite[Section 3]{WO2011} when $b \to (m r_\varepsilon)^+$.
 This completes the verification of the Wang--Young conditions reviewed in Theorem \ref{thm:WY} and concludes the proof of Theorem \ref{main1}. \\
\begin{remark}
The constant $c_0$ in the expansion outside the critical neighborhood (Lemma \ref{lem:misiurewicz_b_full}) is fully compatible with the recovery constant $c_0^{-1}$ and the rate $e^{\frac{1}{3}\lambda_0 p}$ obtained in Lemma \ref{lem:WY_c2_detailed}. Specifically, the expansion gained during the binding phase is sufficient to overcome the initial contraction $d_0$, ensuring that the lower bound for the derivative remains consistent across both the expanding and binding regimes. 
The existence and positive measure of the set $\Delta_\varrho$ are guaranteed by the inductive exclusion argument detailed in Lemma \ref{lem:misiurewicz_a_full}. \\
\end{remark}

\section{Proof of Theorem \ref{main2}}
\label{proof_main_2}

The proof follows the strategy developed by Ott--Wang
\cite[Theorem 3]{OW2010} and Rodrigues \cite{Rodrigues2022}. The key mechanism is that, as
$b\to (mr_\varepsilon)^+,$
one of the critical values winds around the circle
infinitely many times. This winding produces infinitely
many parameter values for which a critical point belongs
to a period-$2$ orbit, yielding superstable sinks.
Assume that \(f_b\) possesses two non-degenerate critical points
\(c_1(b)\) and \(c_2(b)\).
By Lemma~\ref{lem:nondegeneracy}, any such critical point must lie
in an \(\mathcal{O}(\lambda)\)-neighbourhood of either \(0\) or \(\pi\).
Let
$\xi_i(b)
=
f_b(c_i(b)),
i=1,2,
$
denote the corresponding critical values.
Since
$
c_2(b)
=
\pi+\mathcal O(\lambda),
$
one has
$$ \cos c_2(b)
=-1+\mathcal O(\lambda^2).$$
Using the explicit expression of the reduced map,

\[
f_b(\varphi_1)
=
\gamma_b(\varphi_1)
\bigl(
a_{11}\cos\varphi_1
+
a_{12}\sin\varphi_1
\bigr)
-
\frac1\lambda
\ln\!\bigl(
b+m\gamma_b(\varphi_1)\cos\varphi_1
\bigr),
\]
we obtain
\[
\xi_2(b)
=
-\frac1\lambda
\ln\!\bigl(
b-m\gamma_b(c_2(b))
\bigr)
+
\mathcal O(1).
\]
By  Lemma \ref{thm:invariant_graph}, we know that:
$
\gamma_b(c_2(b))
\longrightarrow
r_\varepsilon
$
as
$b\to(mr_\varepsilon)^+.$
Hence
$
b-m\gamma_b(c_2(b))
\longrightarrow0^+.
$
Consequently,
\[
-\frac1\lambda
\ln\!\bigl(
b-m\gamma_b(c_2(b))
\bigr)
\longrightarrow+\infty.
\]

Since the angular variable is defined modulo \(2\pi\), the
critical value
$\xi_2(b)$
winds around the circle $\EU^1$ infinitely many times as
$b\to(mr_\varepsilon)^+$ -- cf. Figure \ref{Imagem 7}.
Define
\[
H(b)
=
f_b(\xi_2(b))
=
f_b(f_b(c_2(b)))
\pmod{2\pi}.
\]
Since the family \(f_b\) depends continuously on the
parameter \(b\), the map
\[
H:(mr_\varepsilon,mr_\varepsilon+\rho)
\longrightarrow
\mathbb S^1
\]
is continuous. 
To make the winding argument rigorous, choose continuous lifts
$\widetilde{\xi}_2(b)$ and $ \widetilde H(b)$
of \(\xi_2(b)\) and \(H(b)\) to the universal cover
\(\mathbb R\).
Because
$
\xi_2(b)
=
-\frac1\lambda
\ln\!\bigl(
b-m\gamma_b(c_2(b))
\bigr)
+
\mathcal O(1),
$
and
$$
b-m\gamma_b(c_2(b))
\longrightarrow0^+,
$$
it follows that
$
\widetilde{\xi}_2(b)
\longrightarrow+\infty
$ as $b\to(mr_\varepsilon)^+.
$
Hence the lifted critical value crosses every interval of
length \(2\pi\) infinitely many times.
Since \(f_b\) is a non-constant circle map, the lifted image
$
\widetilde H(b)
$
also winds around the real line infinitely many times as
\(b\to(mr_\varepsilon)^+\).
On the other hand, the critical point
$
c_2(b)
=
\pi+\mathcal O(\lambda)
$
remains confined to a bounded neighbourhood of \(\pi\).

\begin{figure}[ht]
\begin{center}
 \includegraphics[height=8.5cm]{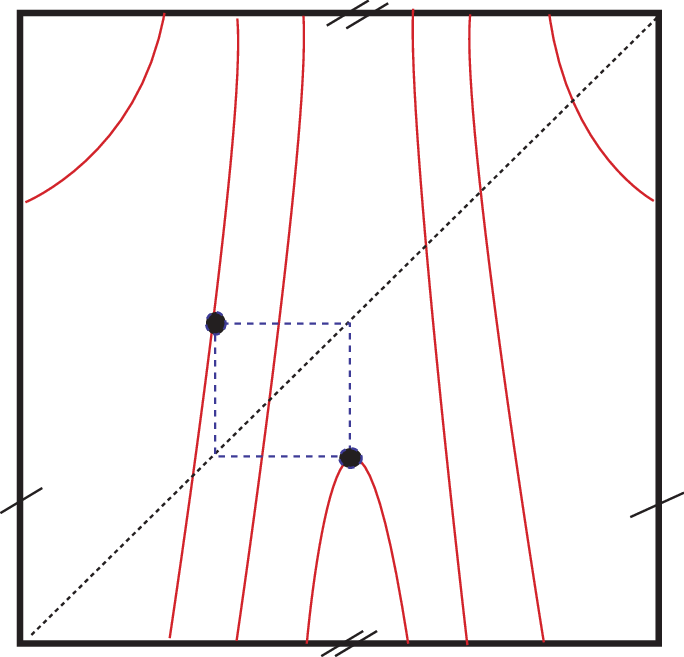}
\end{center}
\caption{\small  Creation of a sink: 2-periodic orbit associated to a critical point of $f_b: \EU^1\to \EU^1$. Bars mean that the sides are identified. }
 \label{Imagem 7}
\end{figure}

Choose a continuous lift
$
\widetilde c_2(b)
$
of \(c_2(b)\).
Define
$
\widetilde G(b)
=
\widetilde H(b)-\widetilde c_2(b).
$
The function \(\widetilde G\) is continuous.
Since \(\widetilde H(b)\) winds around the circle
infinitely many times whereas \(\widetilde c_2(b)\)
remains bounded, the function \(\widetilde G(b)\)
assumes arbitrarily large positive and negative values
in every sufficiently small neighbourhood of
\(mr_\varepsilon\).
Therefore, by the Intermediate Value Theorem,
\(\widetilde G\) possesses infinitely many zeros.
Consequently, there exists a sequence
$
(b_n)_{n\ge1}
$
such that
$
b_n
\longrightarrow
(mr_\varepsilon)^+
$
and
$
\widetilde H(b_n)
=
\widetilde c_2(b_n).
$

Projecting onto the circle yields
$
H(b_n)
=
c_2(b_n)
\pmod{2\pi}.
$
Equivalently, as suggested by Figure \ref{Imagem 7}, one has:
$$
f_{b_n}^2(c_2(b_n))
=
c_2(b_n)
\pmod{2\pi}.
$$
Since both maps
$
b\mapsto \xi_2(b)
$ and $b\mapsto c_2(b)
$
are continuous, and since the relation
$
\xi_2(b)=c_2(b)
$
defines a codimension-one condition, the corresponding
parameter values are isolated.
After discarding such isolated values if necessary, we
may assume that
$
f_{b_n}(c_2(b_n))
\neq
c_2(b_n).
$
Hence
$
\{c_2(b_n),\xi_2(b_n)\}
$
forms a periodic orbit of minimal period \(2\).
The derivative of the second iterate at the critical point is given by:
$$
(f_{b_n}^2)'(c_2(b_n))
=
f'_{b_n}(\xi_2(b_n))
\,f'_{b_n}(c_2(b_n)).
$$
Since \(c_2(b_n)\) is a critical point,
$
f'_{b_n}(c_2(b_n))
=
0.
$
Therefore
$
(f_{b_n}^2)'(c_2(b_n))
=
0.
$
Thus the period-$2$ orbit is superstable for the
one-dimensional map \(f_{b_n}\).
We now lift this periodic orbit to the two-dimensional
return map
$
\mathcal R_b.
$
Let
$$
q_1
=
(\gamma_{b_n}(c_2(b_n)),c_2(b_n))
\qquad \text{and} \qquad
q_2
=
(\gamma_{b_n}(\xi_2(b_n)),\xi_2(b_n))
$$
be the corresponding points on the invariant graph.
Introduce local coordinates
$
(y,\theta),$
where \(y\) is normal to the invariant graph and
\(\theta\) is tangent to it.
Because the invariant graph is \(\mathcal R_b\)-invariant,
the derivative matrix in tangent-normal coordinates has the
triangular form
$$
D\mathcal R_{b_n}(q_i)
=
\begin{pmatrix}
\Lambda_i & 0\\
* & f'_{b_n}(\varphi_i)
\end{pmatrix},
$$
where
$
|\Lambda_i|
$
denotes the contraction rate in the normal direction.
Hence
$$
D\mathcal R_{b_n}^2(q_1)
=
D\mathcal R_{b_n}(q_2)
D\mathcal R_{b_n}(q_1)
=
\begin{pmatrix}
\Lambda_2\Lambda_1 & 0\\
* & 0
\end{pmatrix}.
$$
Theerefore the eigenvalues are
$$
\mu_1=0,
\qquad
\mu_2=\Lambda_1\Lambda_2.
$$
By Lemma~\ref{lem:radial_contraction},
$
|\Lambda_i|
\le
\kappa
<
1.
$
Consequently,
$
|\mu_2|
=
|\Lambda_1\Lambda_2|
\le
\kappa^2
<
1.
$
Thus both eigenvalues lie strictly inside the unit disc:
$$
\mu_1=0,
\qquad
|\mu_2|<1.
$$
Therefore the period-$2$ orbit is a stable sink for the
two-dimensional return map $\mathcal R_{b_n}$. Since one eigenvalue vanishes identically, the sink is
superstable. This proves Theorem~\ref{main2}.

 \section{Discussion and concluding remarks}
\label{sec:concluding_remarks}

In this paper, we have provided a rigorous mathematical proof of the conjecture formulated by Bosch and Simó \cite{BS93} concerning the abundance of ``large'' strange attractors in the unfolding of a Shilnikov--Hopf bifurcation. By verifying the hypotheses of the Wang--Young criterion (reviewed in Theorem \ref{thm:WY}), we have established that the chaotic dynamics observed numerically in earlier studies are a persistent feature of these codimension-two global bifurcations. In terms of the flow \eqref{eq:ivp}, we may conclude that the first return map to a cross-section of $f_{\varepsilon, b} \in \mathcal{X}_{SH}$ may be $C^3$-densely approximated by another exhibiting ``large'' strange attractors or superstable sinks, provided $\varepsilon < \alpha$, $m \ll \lambda \ll 1$, and $b \gtrsim m r_\varepsilon$.

The strange attractors identified herein are  physically observable in the phase space---with basins of attraction of positive Lebesgue measure---and \emph{abundant} in the parameter space. These attractors are statistically robust and support SRB measures, which distinguishes the Shilnikov--Hopf scenario from more transient or fragile forms of chaos. Our findings are valid for the $\mathbb{Z}_2$-symmetric vector field presented by Freire \emph{et al.} \cite{Freire1993}.

The essence of our construction lies in the reduction of the three-dimensional flow \eqref{eq:ivp} to an effective rank-one discrete dynamical system (Proposition \ref{thm:0}). This reduction highlights the intricate interaction between local Hopf instabilities and global homoclinic phenomena. In contrast to the findings of \cite{hom}, we have shown that the resulting attractors are ``large'' in the sense of  \cite{BST98}; they are not located within a small disk-like neighbourhood of a fixed point but instead wind around the unstable manifold $W^u(C_\varepsilon)$ -- this two-dimensional manifold plays the role of the torus of \cite{RodriguesCastro2023, Rodrigues2022}. Since the divergence of \eqref{local_flow} is given by $2\varepsilon + \lambda-4\alpha r^2$, the attractor cannot be arbitrarily close to $W^s_\loc(C_\varepsilon)$ (cf. \cite[pp. 222]{BS93}). 
 Our findings agree well with the numerical results presented in Figures 4, 5, and 8 of \cite{BS93}.

\subsection*{The hypotheses and conditions}

There are two types of conceptual hypotheses involved in this study. Properties \textbf{(P1)--(P4)} define the configuration of the Shilnikov--Hopf bifurcation and have been fundamental in constructing the first return map $\mathcal{R}_b$ of Proposition \ref{thm:0}. They correspond to the Configuration (a) of \cite{BS93}, one of the expected scenarios for the Shilnikov-Hopf bifurcation.
 The non-uniform expansion and the resulting chaotic regimes of $\mathcal{R}_b$ are governed by the interplay of the conditions $m \ll \lambda \ll 1$ and $b \gtrsim m r_\varepsilon$, extensively used in Sections \ref{s: Preparatory} and \ref{s: f_b is M}. Their dynamical implications are summarized as follows:\\

\begin{itemize}
    \item \textbf{Strong radial contraction ($m \ll \lambda$):} This condition constitutes the required contraction for the reduction to a 1D map, as used in Lemmas \ref{lem:radial_contraction} and \ref{prop:reduction}. It ensures that the radial contraction rate towards the invariant curve is significantly faster than the angular expansion. This effectively ``crushes'' the two-dimensional spirals of $\Sigma_{out}$ into the invariant graph $\gamma_b$.\\
    
    \item \textbf{Singular sensitivity ($\lambda \ll 1$):} Small values of $\lambda$ make the ``time of flight'' $\tau \sim \ln(1/|z|)$ (cf. \eqref{Time of Flight}) extremely sensitive to the entry coordinate $z$. This leads to a 1D map with very high derivatives and sharp peaks, facilitating the existence of \emph{Misiurewicz parameters} \cite{WY2006} where the system is purely chaotic and ``sink-free'' for a set of parameters with positive measure (cf. Lemma \ref{lem:misiurewicz_a_full}).\\
    
    \item \textbf{Reinjection geometry ($b \gtrsim m r_\varepsilon$):} This condition is twofold. First, it governs the geometry of reinjection, avoiding homoclinic tangles and leading to the creation of recurrence. Secondly, it allows for sufficient expansion in the first derivative of the reduced map $f_b$.\\
\end{itemize}

Our proof differs from that presented in \cite{Rodrigues2022} whose starting point is a two-dimensional attracting torus. By adding a term that breaks the rotational symmetry, the authors obtained rotational horseshoes and subsequently strange attractors by varying the imaginary part of the saddle-foci. In our case, the  regime  $b \gg m r_\varepsilon$ does not necessarily ensure recurrence. However, as $b \to (m r_\varepsilon)^+$, two-dimensional chaotic structures emerge, shadowing $W^u_{\text{loc}}(C_\varepsilon)$. The dynamical transition from $b \gg m r_\varepsilon$ to $b \to (m r_\varepsilon)^+$ involves the formation of suspended horseshoes as in the ``Torus-Breakdown Theory''. Accompanying this process, saddle-node and period-doubling bifurcations are expected to occur, alongside homoclinic tangencies, Newhouse phenomena, and ``small'' strange attractors as in \cite{Rodrigues2020} for instance.  This has been described by \cite[pp. 227]{BS93}.

\subsection*{Straightforward generalisation}

While we have focused on Configuration (a) of \cite[p. 226]{BS93}, the remaining cases:
\begin{itemize}
\item[(b)] $S_+^1 \to S_+^1$ and $S_-^1 \to S_+^1$;
\item[(c)] $S_+^1 \to S_-^1$ and $S_-^1 \to S_-^1$; and 
\item[(d)] $S_+^1 \to S_-^1$ and $S_-^1 \to S_+^1$
\end{itemize}
may be studied using a similar analytical framework. For cases (b) and (c), it suffices to analyze the map after a single iteration. Case (d) is analytically more involved, as it requires the study of the composition of two maps on the circle, yet it follows the same underlying rank-one principles.

\subsection*{About Theorem \ref{main2}}

A key feature of our analysis is the bridge between chaotic regimes and ordered dynamics. We established the existence of sequences of parameter values $(b_n)_{n \ge 1}$ for which the system possesses superstable periodic orbits. These orbits, characterised by the vanishing of the derivative along the expanding direction at a critical point of $f_b$, accumulate at the boundary of the chaotic regions as $b \to (m r_\varepsilon)^+$. As in \cite{OW2010, Rodrigues2022}, this reinforces the view of rank-one strange attractors. 

The sinks of Theorem  \ref{main2} are conceptually different from those of \cite{BS93}. In their case, sinks arise from the destruction and formation of horseshoes as $b$ varies. In particular, they are associated to saddle-node bifurcations. In our case, they appear as a periodic point of period 2, a technique already used by \cite{hom, OW2010}. The periodic orbits are hyperbolic and persist for ``small'' intervals around $b_n$, the so-called \emph{windows of stability}. 
 \subsection*{Conjecture}

We conclude by addressing the case $b \in (0, m r_\varepsilon)$. If $b + r_1 m \cos \varphi_1 = 0$, then the return map $\mathcal{R}_b$ is not well-defined on the entire annulus $\mathcal{D}$, corresponding to the existence of orbits that fall into $W^s(C_\varepsilon)$ (cf. \cite[p. 222]{BS93}). In this regime, the ellipse $\Psi(W^u_{\text{loc}}(C_\varepsilon)\cap \Sigma_{out})$ intersects the local stable manifold $W^s_{\text{loc}}(C_\varepsilon)\cap \Sigma_{in}$, giving rise to homoclinic tangencies and complex homoclinic tangles \cite{kuznetsov1998elements}. Using the techniques developed in \cite{RodriguesCastro2023, TW2012}, it should be possible to prove a result of the same flavour as Theorem \ref{main1} for this regime. However, such a proof must manage logarithmic singularities that lie outside the domain of the first return map.

\subsection*{Concluding remarks}

The framework established here for the Shilnikov--Hopf configuration underscores the versatility of rank-one theory when addressing concrete global bifurcations. Our analysis suggests that the emergence of ``large'' strange attractors and the existence of SRB measures are ubiquitous features of codimension-two scenarios where strong transverse contraction competes with singular angular expansion. 
As in \cite{BS93}, we have disregarded higher-order terms to focus on the leading-order dynamics of $\mathcal{R}_b$. However, following the same reasoning of \cite[Prop. 2.1]{WY2003}, we conjecture that our main results remain valid if we consider these remainders. 

This study provides a systematic methodology for characterising observable chaos in dissipative systems near a homoclinic network. Building upon the foundational observations in \cite{Argoul1987, Freire1993}, a prominent open challenge remains the analytical location  of Shilnikov--Hopf bifurcations and their associated ``large'' strange attractors within the unfolding of singularities with triple zero eigenvalues (cf. \cite[Eq. (28)]{HK1993}). We defer the exploration of this higher-codimension regime to future research.

Another natural extension would be to investigate the \emph{stochastic stability} of these attractors under small random perturbations, as well as the transition from period-doubling cascades to the observed Misiurewicz maps.

\section*{Declarations}
\subsection*{Data Availability.} The manuscript has no associated data.
\subsection*{Conflict of interest} The author has no Conflict of interest to declare that are relevant to the content of this article.
 \subsection*{Acknowledgments}
The  author is grateful to Carles Sim\'o for giving him this topic of research in 2017 in Barcelona, during a seminar. Also thanks to him for useful comments. Special thanks to Isabel Labouriau for the fruitful discussions. Also special thanks to Lorenzo Díaz and Christian Bonatti for pointing the author out about Theorem C and Section 5 of \cite{DRV}, in University of Oviedo.

\end{document}